\documentclass[reqno]{amsart}
\usepackage{amsthm}
\usepackage{amsfonts}
\usepackage{amssymb}
\usepackage{amsmath}
\usepackage{mathrsfs}
\usepackage{csquotes}

\begin{document}
	%%%%%%%%%%%%%%%%%%%%%%%%%%%%%%%%%%%%%%%%%%%%%%%%%%%%%%%%%%%%%%%%%%%%%%%%%%%%
	\theoremstyle{plain}
	\theoremstyle{definition}

\newtheorem{theorem}{Theorem}[section]
\newtheorem{problem}[theorem]{Problem}
\newtheorem{proposition}[theorem]{Proposition}
\newtheorem{corollary}[theorem]{Corollary}
\newtheorem{lemma}[theorem]{Lemma}
\newtheorem{example}[theorem]{Example}
\newtheorem{remark}[theorem]{Remark}
\newtheorem{definition}[theorem]{Definition}
\newtheorem{fact}[theorem]{Fact}
\newtheorem{claim}[theorem]{Claim}
\newtheorem{note}[theorem]{Note}

	\large
	\title[Rough Weighted Ideal Convergence and Korovkin-Type Applications...]{Rough Weighted Ideal Convergence and Korovkin-Type Approximation via weighted equi-ideal convergence}

	\date{}

	
	\date{}

\author[Tamim Aziz and Sanjoy Ghosal]{Tamim Aziz and Sanjoy Ghosal}
\thanks{Department of Mathematics, University of North Bengal, Raja Rammohunpur, Darjeeling-734013, West Bengal, India. \emph{E-mail}: tamimaziz99@gmail.com (T. Aziz); sanjoykumarghosal@nbu.ac.in; sanjoyghosalju@gmail.com (S. Ghosal). \\ Orcid: https://orcid.org/0009-0004-5727-0121 (T. Aziz);  https://orcid.org/0000-0001-8563-1941 (S. Ghosal)\\	The research of the first author is supported by Human Resource Development Group, CSIR, India, through NET-JRF and the grant number is 09/0285(12636)/2021-EMR-I}
	
	\setcounter {page}{1}

\begin{abstract} If $\omega_t > \beta$ for every $t \in \mathbb{N}$ and for some $\beta > 0$, then the sequence $\{\omega_t\}_{t \in \mathbb{N}}$ represents a weighted sequence of real numbers. In this article, we primarily introduce the concepts of rough weighted ideal limit set and rough weighted ideal cluster points set associated with sequences in normed spaces. Building on these concepts, we derive several important results, including a characterization of maximal ideals, a representation of closed sets in normed spaces, and an analysis of the minimal convergent degree required for the rough weighted ideal limit set to be non-empty. Furthermore, we demonstrate that for an analytic $P$-ideal, the rough weighted ideal limit set forms an $F_{\sigma\delta}$ subset of the normed space. Finally, we introduce the concept of weighted equi-ideal convergence for sequences of functions with respect to analytic $P$-ideals, extending the notion of equi-statistical convergence [Balcerzak et al., J. Math. Anal. Appl. {328} (1) (2007)]. As an application of this notion, we establish a Korovkin-type approximation theorem that serves both as a generalization of [Theorem 2.4, Karaku{\c{s}} et al., J. Math. Anal. Appl. {339} (2) (2008)] and a correction to [Theorem 2.2, Akdağ, Results Math. {72} (3) (2017)].
	\end{abstract}

	\maketitle

	\section{{Introduction}}
	
	\large
	
	We begin with the concept of a submeasure on $\mathbb{N}$ that plays a very pivotal role in this paper. A map $\varphi:\mathcal{P}(\mathbb{N})\to [0,\infty]$ is a submeasure on $\mathbb{N}$ if (i) $\varphi(\varnothing)=0,$ (ii)	if $A\subseteq B,$ then $\varphi(A)\leq \varphi(B),$ (iii)	 $\varphi(A\cup B)\leq \varphi(A)+\varphi(B),$ and (iv) $\varphi(\{t\})<\infty,$ for all $t\in\mathbb{N}.$ A  submeasure  $\varphi$ is  lower semicontinuous   (briefly, lscsm) if $\varphi(A)=\displaystyle\lim_{t\to\infty}\varphi(A\cap[1,t])$ for all $A\subseteq\mathbb{N}.$
	For each lscsm $\varphi$ on $\mathbb{N},$ the exhaustive ideal $Exh(\varphi),$ generated by $\varphi,$ is defined as follows: 
	\begin{equation*}
		Exh(\varphi)=\{A\subset \mathbb{N}:\lim_{t\to\infty}\varphi(A\setminus\{1,2,...,t\})=0\}.
	\end{equation*}
At this stage, we consider an important category of set-theoretical objects, namely ``ideals".
A family $\mathcal{I}\subset\mathcal{P}(\mathbb{N})$ is referred to as an ideal \cite{kostyrko2000convergence} on $\mathbb{N}$ if it fulfills the following conditions:
\begin{itemize}
	\item $\varnothing\in\mathcal{I}$,
	\item if $A,B\in \mathcal{I}$ then $A\cup B\in\mathcal{I}$,
	\item if $A\subset B$ and $B\in\mathcal{I}$ then $A\in \mathcal{I}$.
\end{itemize}
An ideal $\mathcal{I}$ is called non-trivial if $\mathcal{I}\neq\varnothing$ and $\mathcal{I}\neq\mathbb{N}$. We denote by $\mathcal{I}_{fin}$ the ideal of finite subsets of $\mathbb{N}$ and by $\mathcal{I}_\delta=\{A\subseteq \mathbb{N}: \delta(A)=0\}$ the ideal of subsets of $\mathbb{N}$ with natural density zero, where $\delta$ denotes the natural density  \cite{fast1951convergence, fridy1985statistical, vsalat1980statistically}. A non-trivial ideal $\mathcal{I}$ is considered admissible if $\{t\}\in\mathcal{I}$ for each $t\in\mathbb{N}$, i.e., $\mathcal{I}\supseteq \mathcal{I}_{fin}$. %and it is maximal if for each $A\subset\mathbb{N}$ either $A\in\mathcal{I}$ or $\mathbb{N}\setminus A\in\mathcal{I}$. 
Furthermore, an ideal $\mathcal{I}$ is a $P$-ideal \cite{kostyrko2000convergence} if for every sequence $\{A_t\}_{t\in\mathbb{N}}$ of sets in $\mathcal{I}$ there exists  $A\in\mathcal{I}$ such that  $A_t\setminus A$ is finite for each $t\in\mathbb{N}.$ Examples of $P$-ideals include $\mathcal{I}_{fin}$ and $\mathcal{I}_\delta$. A $P$-ideal $\mathcal{I}$ is termed analytic (see \cite{solecki}) if there exists a lscsm $\varphi$ on $\mathbb{N}$ such that
$$
\mathcal{I}=\mathcal{I}_{\varphi}:=\{A\subset\mathbb{N}:\lim_{t\to\infty}\varphi(A\setminus \{1,2,...,t\})=0\}.
$$
 For an ideal $\mathcal{I}$, we will write $\mathcal{F}(\mathcal{I})=\{A\subseteq\mathbb{N}:\mathbb{N}\setminus A\in\mathcal{I}\}$ to denote the dual filter of $\mathcal{I}$.\\

In a different direction, Phu \cite{phu2001rough,xuan2003rough} began an exploration of the theory of rough convergence of a sequence, which serves as an extension of classical convergence in normed spaces, where ``degree of roughness" is recognized as a crucial factor. He also presented several intriguing properties of the set of rough limit points of a sequence in normed spaces, which were quite fascinating and pertinent to this study. Before we proceed, let us formally present the concept of \textit{rough convergence} of a sequence.

\begin{definition}\cite[Page no. 199]{phu2001rough} \cite[Eq (1.1)]{xuan2003rough} Let $r$ be a non-negative real number. A sequence  $\{x_t\}_{t\in \mathbb{N}}$ in a normed space $X$ is said to be  rough convergent  to $x_*$ w.r.t. the degree of roughness $r$ (briefly, $r$-convergence), denoted by $x_t\xrightarrow{r}x_*,$  provided that
	\begin{equation*}~\mbox{for any}~ \varepsilon > 0,~\mbox{there exists}~ t_{\varepsilon} \in \mathbb{N}: t\geq t_{\varepsilon} ~\Rightarrow~ \|x_t - x_*\| \leq  r+ \varepsilon.\end{equation*}
	The  non-negative real number $r$ is known as ``degree of roughness'' and the set $\text{LIM}^r x_i=\left\{x_*\in X: x_t\xrightarrow{r}x_*\right\}$ is called the $r$-limit set of the sequence $\{x_t\}_{t\in \mathbb{N}}.$
\end{definition}
According to \cite[Eq (4.2)]{phu2001rough},  $\tilde{r}(x)\geq 0$ is called minimal degree of roughness or minimal convergent degree of  $x=\{x_t\}_{t\in\mathbb{N}};$  if $x$ is  $r$-convergent for some $r\geq 0,$ then $r\geq \tilde{r}(x)$, otherwise $r<\tilde{r}(x).$ 
For a comprehensive overview of established results on rough convergence, along with relevant references, visit \cite{aytarcore,SGSM2,listan2011characterization,phu2001rough,xuan2003rough}.\\

In 2018, the notion of rough convergence was generalized by employing a broader mode of convergence, namely \textit{rough weighted statistical convergence}. In this context, the idea of a \textit{weighted sequence} and the associated mode of density, referred to as ``weighted density", emerged.
\begin{definition}\cite[Page no. 884]{das2018characterization} \cite[Definition 2]{ghosal2021rough}
Call a sequence $\{\omega_t\}_{t\in\mathbb{N}}$ of real numbers weighted if there exists $\beta>0$ such that $\omega_t>\beta$ for each $t\in \mathbb{N}.$ We denote by $\{\theta_t\}_{t\in\mathbb{N}}$ the sequence of partial sum of $\{\omega_t\}_{t\in\mathbb{N}},$ i.e.,  $\theta_t=\displaystyle\sum_{i=1}^{t}\omega_i.$ Then the weighted density of $A\subseteq\mathbb{N}$ is given by 
	 $$\mathcal{W}\delta(A)=\lim_{t\to\infty}\frac{|A\cap [1, \theta_t]|}{\theta_t},~\mbox{ provided the limit exists}.$$
\end{definition}
It is important to note that (see \cite[Proposition 2]{ghosal2021rough}) the notion of weighted density is quite different from the notion of natural density.

\begin{definition}\cite[Definition 3.3]{das2018characterization}\label{Def 1}\label{defrouwei}
	Let $r\geq 0,$ and let $\{\omega_t\}_{t\in\mathbb{N}}$ be a weighted sequence. Then a real sequence $\{x_t\}_{t\in\mathbb{N}}$  is said to be rough weighted (briefly, $r$-weighted) statistical convergent to $x_*$ w.r.t. the degree of roughness $r$ and the weighted sequence $\{\omega_t\}_{t\in\mathbb{N}}$, denoted by $x_t\xrightarrow[r]{\mathcal{W}st} x_*,$ provided that
$$\mathcal{W}\delta(\{t\in\mathbb{N}:\omega_t|x_t-x_*|> r+\varepsilon\})=0~\mbox{ for any }~\varepsilon>0.$$ 
The collection $\mathcal{W}st-\text{LIM}^rx_t=\left\{x_*\in \mathbb{R}:x_t\xrightarrow[r]{\mathcal{W}st} x_* \right\}$ is called $r$-weighted statistical limit set of $\{x_t\}_{t\in\mathbb{N}}$ in $\mathbb{R}$. 
\begin{itemize}
\item Note that \textit{weighted statistical convergence} arises as a special case when $r=0$ (see \cite[Definition 2.2]{ghosal2016weighted} with $\alpha=1$).
\item One may also observe that the definition of statistical convergence (see \cite{steinhaus1951convergence,tripathy1997statistically}) is reached when $r=0$ and $\omega_t=1$ for all $t\in\mathbb{N}$.
\end{itemize}
For further reading along these lines, the reader is referred to \cite{gho2019, Ghosal2020}, which explore various aspects of rough weighted statistical convergence, and related summability methods.
\end{definition}
In 2007, Balcerzak et al. \cite{balcerzak2007statistical} extended the notion of statistical convergence to the concept of equi-statistical convergence for sequences of functions. Before proceeding with the definition, 
let $C(K)$ denote the space of all continuous real-valued functions defined on a compact subset
$K(\subseteq\mathbb{R})$. It is well-known that $C(K)$ is a Banach space when equipped with the norm
$$
\|f\|_{C(K)}=\sup_{x\in K}|f(x)|,~f\in C(K).
$$

%thereby broadening its scope to encompass the earlier concept of equi-statistical convergence introduced 
\begin{definition}\cite[Page no. 719]{balcerzak2007statistical}
A sequence of functions $\{f_t\}_{t \in \mathbb{N}}$ is said to be equi-statistically convergent to $f \in C(K)$ if, for any $\varepsilon > 0$, the sequence $\{g_{j,\varepsilon}\}_{j \in \mathbb{N}}$ of functions $g_{j,\varepsilon} : K \to \mathbb{R}$, defined by
\begin{equation*}
	g_{j,\varepsilon}(x) = \frac{|\{t \in \mathbb{N} : |f_t(x) - f(x)| > \varepsilon\} \cap [1,j]|}{j}, \quad x \in K,
\end{equation*}
converges uniformly to the zero function on $K$.
\end{definition}
Karakuş et al. \cite{karakucs2008equi} established a Korovkin-type approximation theorem using the concept of equi-statistical convergence. This result pertains to sequences of positive linear operators acting on $C(K)$.
\begin{theorem}\cite[Theorem 2.4]{karakucs2008equi}\label{equistcon}
	Let $\{L_t\}_{t\in\mathbb{N}}$ be a sequence of positive linear operators from $C(K)$ into $C(K)$. Then, for all $f \in C(K)$,
	$$
	L_t(f) \to f ~~\quad(equi-st ) ~\text{ on } K
	$$
	if and only if
	\begin{equation*}
		L_t(e_i) \to e_i ~~\quad(equi-st )~ \text{ on } K,
	\end{equation*}
	where $e_i(x) = x^i$ for $i = 0, 1, 2$. 
\end{theorem}
From another perspective, Akda{\u{g}} \cite{akdaug2017weighted} introduced the concept of weighted equi-statistical convergnce for sequences of functions, encompassing the earlier notion of equi-statistical convergence, as defined below: 

\begin{definition}\cite[Definition 2.2]{akdaug2017weighted}
	A sequence of functions $\{f_t\}_{t\in\mathbb{N}}$ is said to be weighted equi-statistically convergent (briefly, $\omega-equi-st$) to $f$  on $K(\subseteq \mathbb{R})$ w.r.t. the weighted sequence $\{\omega_t\}_{t\in\mathbb{N}}$ if, for any $\varepsilon>0$, the sequence $\{g_{j,\varepsilon}\}_{j\in\mathbb{N}}$ of functions $g_{j,\varepsilon}:K\to \mathbb{R}$ given by 
	\begin{equation*}
	g_{j,\varepsilon}(x)=\frac{|\{t\leq \theta_j:\omega_t|f_t(x)-f(x)|> \varepsilon\}|}{\theta_j},~x\in K
	\end{equation*}
	is uniformly convergent to zero function on $K$. It is evident that this notion coincides with the concept of equi-statistical convergence  when $\omega_t=1$ for each $t\in\mathbb{N}$.
\end{definition}
Building on this notion, Akdağ \cite{akdaug2017weighted} formulated a weighted Korovkin-type approximation theorem, presented as follows:
\begin{theorem}\cite[Theorem 2.2]{akdaug2017weighted}\label{thMm5.4}
	Let $\{L_t\}_{t\in\mathbb{N}}$ be a sequence of positive linear operators from $C(K)$ into $C(K)$. Then, for all $f \in C(K)$,
	\[
	L_t(f) \to f ~~\quad(\omega-equi-st ) ~\text{ on } K
	\]
	if and only if
	\begin{equation*}
		L_t(e_i) \to e_i ~~\quad(\omega-equi-st )~ \text{ on } K.
	\end{equation*}
	
\end{theorem}
\begin{note}
	In the proof of the above theorem, it is shown (see after Eq (2.5) in \cite[Theorem 2.2]{akdaug2017weighted}) that
	$$
	|L_t(f;x)-f(x)|<\varepsilon+B(x,\varepsilon)\sum_{i=0}^{2}|L_t(e_i;x)-e_i(x)|,
	$$
	where $B(x,\varepsilon)=\varepsilon + C + \frac{2C} {\delta^2 }(\|x\|^2 + 2\|x\| + 1)$ and $C=\displaystyle\sup_{x\in K}|f(x)|$. By multiplying both sides by $\omega_t$, we obtain 
	$$
	\omega_t|L_t(f;x)-f(x)|<\varepsilon\omega_t+B(x,\varepsilon)\sum_{i=0}^{2}\omega_t|L_t(e_i;x)-e_i(x)|,
	$$
	Then, due to the arbitrariness of $\varepsilon$, it is assumed that 
	$$
	|L_t(f;x)-f(x)|\leq B(x,\varepsilon)\sum_{i=0}^{2}\omega_t|L_t(e_i;x)-e_i(x)|,
	$$
	i.e., the term $\varepsilon\omega_t$ is neglected. However, such an omission is not justified when $\omega_t\to\infty$ as $t\to\infty$.
\end{note}
After that, the following example is considered, and it too is found to be incorrect. 
\begin{example}\cite[Example 2.2]{akdaug2017weighted}\label{example5.5}
	For each $t\in \mathbb{N}$, define the function $h_t \in C[0,1]$ by
	$$
	h_t(x) = 
	\begin{cases}
		\displaystyle 2^{t+1}(x - \frac{1}{2^t}), & \text{if } x \in \left[\frac{1}{2^t},  \frac{1}{2^{t-1}}-\frac{1}{2^{t+1}}\right], \\
		-2^{t+1}\left(x - \frac{1}{2^{t-1}}\right), & \text{if } x \in \left[\frac{1}{2^{t-1}}-\frac{1}{2^{t+1}}, \frac{1}{2^{t-1}}\right],\\
		0, & \text{otherwise}.
	\end{cases}
	$$
	Consider the classical Bernstein operators 
	$$
	B_t(f;x) = \sum_{k=0}^{t} f\left( \frac{k}{t} \right) \binom{t}{k} x^k (1 - x)^{t - k}, \quad x\in [0,1]~\mbox{ and }~f\in C[0,1].
	$$
	
	Set $\omega_t = \sqrt{t}$ for each $t \in\mathbb{N}$. Then, the sequence   $\{L_t : C[0,1] \to C[0,1]\}_{t \in \mathbb{N}}$ of positive linear operators, defined by
	$$
	L_t(f;x)=(1+h_t(x))B_t(f;x),~\quad x\in [0,1]~\mbox{ and }~f\in C[0,1],
	$$
	satisfies the following property:
	$$
	L_t(f;x)\to f~~\quad(\omega-equi-st )~\mbox{ for all }~f\in C[0,1].
	$$
	Unfortunaltely, this assertion is incorrect. Later in Section \ref{section 5}, we will demonstrate that  there exists a class of functions $f\in C[0,1]$ for which this assertion does not hold, thereby also suggesting that  Theorem \ref{thMm5.4} is not generally valid.
\end{example}
At this point, we recall a significant result that will be instrumental for our purposes.

\begin{theorem}\cite[Theorem 2]{tachev2012complete}\label{tachev}
	If $q \in \mathbb{N}$ is odd and $f \in C^q[0, 1]$, then uniformly in $x \in [0, 1]$, as $t \to \infty$, we have
	$$
	t^{\frac{q}{2}} \cdot \left( B_t(f; x) - f(x) - \sum_{r=1}^{q} B_t\left((e_1 - x)^r; x\right) \cdot \frac{f^{(r)}(x)}{r!} \right) \to 0.
	$$
\end{theorem}

We are now in a position to introduce our main definitions, namely Definitions~\ref{Def 2} and \ref{defIequi}, which serve to generalize the concepts of rough weighted statistical convergence for sequences and weighted equi-statistical convergence for sequences of functions, respectively
\begin{definition}\label{Def 2}
Let $r\geq 0$, $\{\omega_t\}_{t\in\mathbb{N}}$ be a weighted sequence and $\mathcal{I}$ be an ideal on $\mathbb{N}$. A sequence $\{x_t\}_{t\in\mathbb{N}}$ with values in a normed space $X$ is said to be rough weighted  $\mathcal{I}$-convergent to $x_*$ w.r.t. the degree of roughness $r$ and the weighted sequence $\{\omega_t\}_{t\in\mathbb{N}}$ (briefly, $r-(\omega_t,\mathcal{I})$-convergence), denoted by $x_t\xrightarrow[r]{(\omega_t,\mathcal{I})} x_*,$ provided that
$$\{t\in\mathbb{N}:\omega_t\|x_t-x_*\|> r+\varepsilon\}\in\mathcal{I},~\mbox{for any}~\varepsilon>0.$$ 
We denote by $(\omega_t,\mathcal{I})-\text{LIM}^rx_t=\left\{x_*\in X:x_t\xrightarrow[r]{(\omega_t,\mathcal{I})} x_* \right\}$ the $r$-weighted $\mathcal{I}$-limit set of $\{x_t\}_{t\in\mathbb{N}}$ in $X$.
\end{definition}

%Note that:
 \begin{remark} Note that
 	\begin{itemize}
\item [(i)] if $\omega_t=1$ for all $t\in\mathbb{N}$  then $r$-$(\omega_t,\mathcal{I})$-convergence corrosponds to rough $\mathcal{I}$-convergence (see \cite{dundar2014rough,Dundar2016}). Furthermore, if $\mathcal{I}=\mathcal{I}_{\delta}$ (resp. $\mathcal{I}=\mathcal{I}_{fin}$)   then $r-(\omega_t,\mathcal{I})$-convergence conincides with   $r$-statistical convergence  (see \cite{aytar2008rough}) (resp. $r$-convergence).
\item [(ii)] if $\mathcal{I}=\{A\subset\mathbb{N}:\mathcal{W}\delta(A)=0\}$ then $r$-$(\omega_t,\mathcal{I})$-convergence simplifies to $r$-weighted statistical convergence (see \cite{das2018characterization,ghosal2021rough}).
\end{itemize}
 \end{remark}

\begin{definition}
A weighted sequence $\{\omega_t\}_{t\in\mathbb{N}}$ is said to be $\mathcal{I}$-bounded if there exists $\mu > 0$ such that  $\{t \in \mathbb{N} : \omega_t > \mu\}\in\mathcal{I}$. Alternatively, this can be expressed as  $\{t \in \mathbb{N} : \omega_t < \mu\}\in\mathcal{F}(\mathcal{I})$.
\end{definition}
\begin{definition}\label{defIequi}
	Let $\mathcal{I}_{\varphi}$ be an analytic $P$-ideal, $\{\omega_t\}_{t\in\mathbb{N}}$ be a weighted sequence and $\{f_t\}_{t\in\mathbb{N}}$ be a sequence of functions in $C(K)$. Then $\{f_t\}_{t\in\mathbb{N}}$ is said to be weighted equi-$\mathcal{I}_{\varphi}$ (briefly, $\omega-equi-\mathcal{I}_{\varphi}$) convergent to $f$ on $K$ w.r.t. the weighted sequence $\{\omega_t\}_{t\in\mathbb{N}}$, if for every $\varepsilon>0$, the sequence $\{h_{j,\varepsilon}\}_{j\in\mathbb{N}}$ of functions $h_{j,\varepsilon}:K\to \mathbb{R}$ given by 
	\begin{equation}\label{equiideal}
		h_{j,\varepsilon}(x)=\varphi(\{t\in\mathbb{N}:\omega_t|f_t(x)-f(x)|> \varepsilon\}\setminus[1,j]),~x\in K
	\end{equation}
	is uniformly convergent to zero function on $K$.

 It can be shown that  $\omega-equi-\mathcal{I}_{\varphi}$ convergence simplifies to $\omega-equi-st$ convergence if one defines $$\varphi(A)=\sup_{t\in\mathbb{N}}\frac{|A\cap [1,\theta_t]|}{\theta_t}~\mbox{ for all }~A\subseteq \mathbb{N}.$$ Furthermore, if $\omega_t=1$ for each $t\in\mathbb{N}$, then setting 
 $$\varphi(A)=\sup_{k\in\mathbb{N}}\frac{|A\cap [1,t]|}{t}~\mbox{ for all }~A\subseteq \mathbb{N},$$
leads to the definition of equi-statistical convergence 
(see Section~\ref{section 5} for details).
\end{definition}

In this article, we primarily focus on the ``rough weighted ideal limit set" associated with a sequence in normed spaces, i.e., $(\omega_t, \mathcal{I})-\text{LIM}^r x_t$. Our goal is to deepen the understanding of the nature of this rough weighted ideal limit set with ``minimal convergent degree" within normed spaces. Along the way, the concept of ``rough weighted ideal cluster points set" of a sequence in normed spaces is introduced. We use this concept to demonstrate several notable results that include a characterization of maximal ideals and a representation of closed sets within normed spaces. Finally, we establish a ``Korovkin-type approximation theorem concerning weighted equi-ideal convergence" for sequences of functions relative to analytic $P$-ideals, thereby providing both as a correction of Theorem \ref{thMm5.4} and   a generalization of Theorem \ref{equistcon}.\\

The remaining portion of the paper consists of four sections. In Section \ref{sec 2}, we examine the newly proposed concept of ``rough weighted ideal limit set" and obtain our key findings as follows:
\begin{itemize}
\item
The rough weighted ideal limit set $(\omega_t,\mathcal{I})-\text{LIM}^r x_t$ of a sequence $\{x_t\}_{t\in\mathbb{N}}$ in a normed space $X$ is an $F_{\sigma\delta}$ (and hence a Borel) subset of $X$, provided that $\mathcal{I}$ is an analytic $P$-ideal [Proposition \ref{th2.3}].
\item If $\{\omega_t\}_{t\in\mathbb{N}}$ is not $\mathcal{I}$-bounded then  $(\omega_t,\mathcal{I})-\text{LIM}^r x_t$  contains at most one element [Lemma \ref{le 1}].
\item  The set $(\omega_t,\mathcal{I})-\text{LIM}^r x_t$ is closed, bounded, and convex in any dimensional normed space $X$ [Lemmas \ref{le 2} and \ref{le3}]; however it may not be compact in $X$ [Note \ref{note 2.2}].
\item For any sequence $\{x_t\}_{t\in\mathbb{N}}$ in a uniformly convex normed space \(X\), the set $(\omega_t, \mathcal{I})-\text{LIM}^r x_t$ is strictly convex, given that $\mathcal{I}$ is a $P$-ideal [Theorem \ref{strict convex}]; the condition imposed on the normed space $X$ cannot, in general, be relaxed [Example \ref {ex 2.12}].
\end{itemize}
In Section \ref{sec 3}, we introduce the idea of the ``{minimal weighted ideal convergent degree" $\tilde{r}(x)$ of a sequence $x=\{x_t\}_{t\in\mathbb{N}}$ w.r.t. a given weighted sequence and a given ideal. Based on this concept, we have established the following:
\begin{itemize}
	\item If $\{\omega_t\}_{t\in\mathbb{N}}$ is $\mathcal{I}$-bounded then $int((\omega_t,\mathcal{I})-\text{LIM}^{r} x_t)\neq \varnothing$ when $r>\tilde{r}(x)$ [Proposition \ref{Th 2.4}]. 
	\item 
	If  $x=\{x_t\}_{t\in\mathbb{N}}$ is contained in a totally bounded subset of $X$ then $int((\omega_t,\mathcal{I})-\text{LIM}^{\tilde{r}(x)} x_t)=\varnothing$ [Proposition \ref{Pro 3}].
	\item For any sequence $x=\{x_t\}_{t\in\mathbb{N}}$ in a reflexive normed space $X$, we have $(\omega_t,\mathcal{I})-\text{LIM}^{\tilde{r}(x)} x_t\neq\varnothing$ [Theorem \ref{Th 2.6}]; however the reflexivity of $X$ cannot, in general, be neglected [Example \ref{non-reflexive}].
	\item 
	For any sequence $x=\{x_t\}_{t\in\mathbb{N}}$  in a uniformly convex Banach space $X$.
Then:
	\begin{itemize}
		\item[(i)] If $(\omega_t,\mathcal{I})-\text{LIM}^{r} x_t$ is singletone and $\{\omega_t\}_{t\in\mathbb{N}}$ is $\mathcal{I}$-bounded  then $r=\tilde{r}(x)$.
		\item [(ii)] If $r=\tilde{r}(x)$ and $\mathcal{I}$ is a $P$-ideal then $(\omega_t,\mathcal{I})-\text{LIM}^{r} x_t$ is singleton.
	\end{itemize}
[Theorem \ref{Th 2.14}].
\end{itemize} 
Section \ref{sec 4} focuses on the study on ``rough weighted ideal cluster points set" of a sequence $\{x_t\}_{t\in\mathbb{N}}$, denoted by $(\omega_t,\mathcal{I})-\Gamma_{x_t}^r$, in normed spaces. This section includes the following significant findings: 
\begin{itemize}
\item The set $(\omega_t,\mathcal{I})-\Gamma_{x_t}^r$ may or may not be closed in $X$ [Example \ref{ex 2}].
\item The set $(\omega_t,\mathcal{I})-\Gamma_{x_t}^r$ is closed in $X$ if  $\{\omega_t\}_{t\in\mathbb{N}}$ is $\mathcal{I}$-bounded for any ideal $\mathcal{I}$ on $\mathbb{N}$ [Proposition \ref{closedness cluster}] or if $\mathcal{I}$ is a maximal admissible ideal [Proposition \ref{maximal closed cluster}].
\item For a maximal admissible ideal $\mathcal{I}$, we have $(\omega_t,\mathcal{I})-\text{LIM}^{r} x_t=(\omega_t,\mathcal{I})-\Gamma_{x_t}^r$ for any sequence $\{x_t\}_{t\in\mathbb{N}}$ in $X$ and vice versa [Theorem \ref{maximal ideal}].
\item Each non-empty closed set $F$ in a separable normed space $X$ has a represented in terms of  $(\omega_t,\mathcal{I})-\Gamma_{x_t}^r$ for some sequence $\{x_t\}_{t\in\mathbb{N}}$ in $X$, if the ideal $\mathcal{I}$  possesses a certain property [Theorem \ref{closed set}].
\end{itemize} 
Section \ref{section 5} is devoted to the study of Korovkin-type approximation theorem via weighted equi-ideal convergence for sequences of functions with respect to analytic $P$-ideals. The section covers the following essential results:
\begin{itemize}
	\item  For a particular choice of the submeasure $\varphi$, weighted equi-$\mathcal{I}_{\varphi}$ convergence coincides with weighted equi-statistical convergence [Proposition \ref{procoincidence}].
	\item  By invoking Theorem \ref{tachev}, we conclude that both Theorem \ref{thMm5.4} and Example \ref{example5.5} are incorrect
	 [Example \ref{exwrong}].
	 \item A Korovkin-type approximation theorem via weighted equi-$\mathcal{I}_{\varphi}$ for sequences of functions is examined, which revises Theorem \ref{thMm5.4} and generalizes Theorem \ref{equistcon} [Theorem \ref{korokvingequiideal}].
\end{itemize}
\vspace{0.2cm}

\begin{itemize}
\item[\textbf{Notations:}] Throughout the paper,  $X$ represents a normed space, $B_X$ denotes the closed unit ball centre at $0\in X$, $\{x_t\}_{t \in \mathbb{N}}$ is a sequence of elements in $X$, and $\{\omega_t\}_{t \in \mathbb{N}}$ signifies a weighted sequence. 
\end{itemize}

\section{Some characterization of the set $(\omega_t,\mathcal{I})-\text{LIM}^r x_t$. }\label{sec 2}

Our first result in this section portrays that the limit set $(\omega_t,\mathcal{I})-\text{LIM}^r x_t$ is a Borel subset of $X$, meaning that it is topologically nice for a wider class of ideals. 
\begin{proposition} \label{th2.3}
For an analytic $P$-ideal $\mathcal{I}$, the limit set $(\omega_t,\mathcal{I})-\text{LIM}^r x_t$ is an $F_{\sigma\delta}$ (hence a Borel) subset of $X.$ 
\end{proposition}
\begin{proof}
	Since $\mathcal{I}$ is an analytic $P$-ideal, we have $\mathcal{I}=\mathcal{I}_{\varphi}$ for some  lscsm $\varphi$ on $\mathbb{N}$.
	For $k,t\in\mathbb{N}$, we consider the open set $$U_{t,k}=\left\{x_*\in X:\omega_t\|x_t-x_*\|>r+\frac{1}{k}\right\}.$$
	Note that the definition of rough $(\omega_t,\mathcal{I})$-convergence ensures that 
	\begin{align}\label{eq 1a}
		&(\omega_t,\mathcal{I})-\text{LIM}^r x_t\nonumber\\
		=& \left\{x_*\in X: (\forall k\in\mathbb{N})~ \{t\in\mathbb{N}:x_*\in U_{t,k} \}\in\mathcal{I}_{\varphi}\right\}\nonumber\\
		=&\bigcap_{k=1}^{\infty}\left\{x_*\in X:  \{t\in\mathbb{N}:x_*\in U_{t,k} \}\in\mathcal{I}_{\varphi}\right\}\nonumber\\
		=&\bigcap_{k=1}^{\infty}\left\{x_*\in X: \lim_{t\to\infty} \varphi \left(\{i\in\mathbb{N}:x_*\in U_{i,k} \}\setminus\{1,2,...,t\}\right)=0 \right\}\nonumber\\
		=&\bigcap_{k=1}^{\infty}\left\{x_*\in X: (\forall j\in\mathbb{N})~(\exists m\in\mathbb{N}): \varphi(\{i\in\mathbb{N}:x_*\in U_{i,k} \}\setminus\{1,2,...,t\})\leq\frac{1}{j}~(\forall t\geq m) \right\}\nonumber\\
		=&\bigcap_{k=1}^{\infty}\bigcap_{j=1}^{\infty}\bigcup_{m=1}^{\infty}\bigcap_{t=m}^{\infty} \left\{x_*\in X:  \varphi(\{i\in\mathbb{N}:x_*\in U_{i,k} \}\setminus\{1,2,...,t\})\leq\frac{1}{j}\right\}
		\end{align}
	For every fixed $t,j\in\mathbb{N},$ consider the family  $$\mathscr{F}_{t,j}=\{F\subseteq \mathbb{N}\setminus\{1,2,..,t\}: \varphi{(F)}>\frac{1}{j}\}.$$
	Observe that for each \( F \in \mathscr{F}_{t,j} \), we have \( \varphi(F) > \frac{1}{j} \). Since \( \varphi \) is lower semicontinuous, we can express \( \varphi(F) \) as \( \lim_{i \to \infty} \varphi(F \cap [1,i]) \). This means there exists some \( i_0 \in \mathbb{N} \) such that \( \varphi(F \cap [1,i_0]) > \frac{1}{j} \). By well ordring prperty we choose $i_F=\min\{i\in\mathbb{N}: \varphi(F \cap [1,i]) > \frac{1}{j}\}$. Note that $\varphi(F \cap [1,i_F]) > \frac{1}{j}$.
	Thus, it is sufficient to consider \( F \subseteq [1,i_F] \). Therefore, without loss of generality, we can assume that each \( F \in \mathscr{F}_{t,j} \) is finite.
	Then, notice that 
	\begin{align*}
		B_{k,j,t}=&\left\{x_*\in X:  \varphi(\{i\in\mathbb{N}:x_*\in U_{i,k} \}\setminus\{1,2,...,t\})\leq\frac{1}{j}\right\}\\
		=&\left\{x_*\in X: (\forall F\in \mathscr{F}_{t,j})~(\exists i\in F)~\mbox{such that}~x_*\notin U_{i,k}\right\}\\
		=& \bigcap_{F\in \mathscr{F}_{t,j}} \bigcup_{i\in F} X\setminus U_{i,k}.
	\end{align*}
	Since $U_{i,k}$ is an open set for every $i,k\in\mathbb{N}$ and $F$ is finite, it follows that  $B_{k,j,t}$  is a closed subset of $X.$ Finally, in view of Eq (\ref{eq 1a}), we deduce that $(\omega_t,\mathcal{I})-\text{LIM}^r x_t$ is an $F_{\sigma\delta}$ subset of $X.$
	\end{proof}

\begin{lemma}\label{le 1}
Suppose $r\geq 0$ and $\{\omega_t\}_{t\in\mathbb{N}}$ is not $\mathcal{I}$-bounded. Then for every $\{x_t\}_{t\in\mathbb{N}}$ in $X$, $(\omega_t,\mathcal{I})-\text{LIM}^r x_t$ is either empty or singletone.
\end{lemma}
\begin{proof}
	Assume, on the contrary, that there exist $x_*,y_*\in X$ with $x_*\neq y_*$ such that $x_*,y_*\in (\omega_t,\mathcal{I})-\text{LIM}^r x_t$. We set $\varepsilon=\frac{\|x_*-y_*\|}{2}>0$. Let us consider $A=\{t\in\mathbb{N}:\omega_t>\frac{2(r+\varepsilon)}{\|x_*-y_*\|}\}$. Note that  $A\notin\mathcal{I}$ since  $\{\omega_t\}_{t\in\mathbb{N}}$ is not $\mathcal{I}$-bounded. Now, observe that
	\begin{eqnarray*}
		A&=&\{t\in\mathbb{N}:\omega_t\|x_*-y_*\|>2(r+\varepsilon)\}\\
		&\subseteq& \{t\in\mathbb{N}:\omega_t\|x_t-x_*\|>r+\varepsilon\}\cup \{t\in\mathbb{N}:\omega_t\|x_t-y_*\|>r+\varepsilon\}\in\mathcal{I}.
	\end{eqnarray*}
This contradicts the fact that $A\notin\mathcal{I}$. Therefore, we can conclude that  $(\omega_t,\mathcal{I})-\text{LIM}^r x_t$ is atmost singletone.
\end{proof}
\begin{lemma}\label{le 2}
For a sequence $\{x_t\}_{t\in\mathbb{N}}$ in $X$,
$$
\mbox{diam}((\omega_t,\mathcal{I})-\text{LIM}^r x_t)\leq \frac{2r}{\displaystyle\liminf_{t\in\mathbb{N}}\omega_t}.
$$
\end{lemma}
\begin{proof}
	Assume, on the contrary, that $\mbox{diam}((\omega_t,\mathcal{I})-\text{LIM}^r x_t)> \frac{2r}{\displaystyle\liminf_{t\in\mathbb{N}}\omega_t}$. Then, there exists $0<\alpha<\displaystyle\liminf_{t\in\mathbb{N}}\omega_t$ such that $\mbox{diam}((\omega_t,\mathcal{I})-\text{LIM}^r x_t)> \frac{2r}{\alpha}$. Now, we can find $x_*,y_*\in (\omega_t,\mathcal{I})-\text{LIM}^r x_t$ such that $\|x_*-y_*\|>\frac{2r}{\alpha}$. Therefore, by Lemma \ref{le 1}, $\{\omega_t\}_{t\in\mathbb{N}}$ is $\mathcal{I}$-bounded. So there exists $\mu>0$ such that $\{t\in\mathbb{N}:\omega_t<\mu\}\in\mathcal{F}(\mathcal{I})$. Since $\displaystyle\liminf_{t\in\mathbb{N}}\omega_t>\alpha$, there exists $t_0\in\mathbb{N}$ such  that $\omega_t>\alpha$ for all $t>t_0$. We now fix $0<\varepsilon< \frac{\alpha\|x_*-y_*\|}{2}-r$.   Therefore, observe that
	\begin{eqnarray*}
	A&=&\{t\in\mathbb{N}:\omega_t<\mu\}\setminus\{1,2,...,t_0\}\in \mathcal{F}(\mathcal{I}),\\
	B&=&\{t\in\mathbb{N}:\omega_t\|x_t-x_*\|<r+\varepsilon\}\in \mathcal{F}(\mathcal{I}),\\
		\mbox{  and,  }~
	C&=&\{t\in\mathbb{N}:\omega_t\|x_t-y_*\|<r+\varepsilon\}\in \mathcal{F}(\mathcal{I}).
	\end{eqnarray*}
This ensures that $A\cap B\cap C\in \mathcal{F}(\mathcal{I})$. Pick any $t\in A\cap B\cap C$. Then, we get that 
\begin{eqnarray*}
\alpha\|x_*-y_*\|&\leq &\alpha\|x_t-x_*\|+\alpha\|x_t-y_*\|\\
&<&\omega_t\|x_t-x_*\|+\omega_t\|x_t-y_*\|<2(r+\varepsilon)<\alpha\|x_*-y_*\|,
\end{eqnarray*}
$-$which is a contradiction. Therefore, we can infer that
$\mbox{diam}((\omega_t,\mathcal{I})-\text{LIM}^r x_t)\leq \frac{2r}{\displaystyle\liminf_{t\in\mathbb{N}}\omega_t}.$
\end{proof}
\begin{lemma}\label{le3}
For a sequence $\{x_t\}_{t\in\mathbb{N}}$ in $X$, $(\omega_t,\mathcal{I})-\text{LIM}^r x_t$ is a closed and convex subset of $X$.
\end{lemma}
\begin{proof}
The proof is similar to \cite[Theorems 3.3 and 3.4]{das2018characterization}. So is omitted.
\end{proof}
\begin{note}\label{note 2.2}
Lemmas \ref{le 2} and \ref{le3} show that $(\omega_t,\mathcal{I})-\text{LIM}^r x_t$ is a closed and bounded subset of $X$. However, in general, it need not be a compact subset of $X$. For instance, consider the normed space $\ell^\infty(\mathbb{R})$ of real bounded sequences equipped with $sup$ norm. 
Now consider the sequence $\{e_t\}_{t\in\mathbb{N}},$ where $e_j$ denotes the sequence $\{0,...,0,1,0,...\}$ with the $j$-th place is $1$ and the remaining positions are $0.$
Next, we set
\begin{equation*}
	\omega_t=\beta~\mbox{ where }~t\in\mathbb{N}~\mbox{ and }~\beta>0.
\end{equation*}
  Then it is evident that
$$
\{e_1,e_2,...,e_k,...\}\subseteq (\omega_t,\mathcal{I})-\text{LIM}^{\beta} e_t~\mbox{ for each non-trivial ideal }~\mathcal{I}.
$$
A closed subset of a compact set is also compact. Consequently, $(\omega_t,\mathcal{I})-\text{LIM}^{\beta} e_t
$ is not compact in $\ell^\infty(\mathbb{R}).$
\end{note}

\begin{proposition}\label{th 2.4}
	Suppose $r,\sigma \ge 0$ and $\{\omega_t\}_{t\in\mathbb{N}}$ is $\mathcal{I}$-bounded. Then for every $\{x_t\}_{t\in\mathbb{N}}$ in $X,$ there exists  $\mu>0$ such that
	\begin{eqnarray*}(\omega_t,\mathcal{I})-\text{LIM}^r x_t + \sigma {B}_X\subseteq (\omega_t,\mathcal{I})-\text{LIM}^{r+\sigma \mu}x_t.
	\end{eqnarray*}
\end{proposition}
\begin{proof}
	Assume that $\{\omega_t\}_{t\in\mathbb{N}}$ is $\mathcal{I}$-bounded. Then there exists $\mu>0$ such that 
$$
A=\{t\in \mathbb{N} : \omega_t>\mu\}\in\mathcal{I}.
$$  
Choose any $y_*\in(\omega_t,\mathcal{I})-\text{LIM}^r x_t + \sigma {B}_X$. Then, $y_*=x_*+x$ for some  $x_*\in(\omega_t,\mathcal{I})-\text{LIM}^r x_t$ and some $x\in \sigma {B}_X$. Let $\varepsilon>0$ be given. First, observe that
 $$
 B=\{t\in \mathbb{N} : \omega_t\|x_t-x_*\|>r+\varepsilon\}\in\mathcal{I}.
 $$ 
 Now, observe that 
	 $$
	 \mathbb{N}\setminus(A\cup B)\subseteq \{t\in \mathbb{N} : \omega_t\|x_t-y_*\|<r+\sigma \mu+\varepsilon\}.
	 $$
	As $A\cup B\in\mathcal{I}$, we get that $\{t\in \mathbb{N} : \omega_t\|x_t-y_*\|> r+\sigma \mu+\varepsilon\}\in\mathcal{I}$.
	 Since $\varepsilon>0$ was arbitrary, we deduce that $y_*\in(\omega_t,\mathcal{I})-\text{LIM}^{r+\sigma \mu}x_t$. Hence the result.\end{proof}
 
The requirement that the weighted sequence $\{\omega_t\}_{t\in\mathbb{N}}$ is $\mathcal{I}$-bounded in Proposition \ref{th 2.4} cannot be generally weakened. Let us present an example to reinforce this point.

\begin{example}
 Consider the sequence  $\{x_t\}_{t\in\mathbb{N}}$ and the weighted sequence $\{\omega_t\}_{t\in\mathbb{N}},$ which are  defined, respectively, as follows:
	\begin{align*}
		x_{t}=
		\begin{cases}
			\frac{(-1)^t}{t} + \frac{1}{t^2} &\mbox{if}~ t\notin \{m^2:m\in\mathbb{N}\},\\
			t^2&\text{otherwise,} 
		\end{cases}
		~~\text{ and }~~~
		\omega_{t}=
		%\begin{cases}
		%	\frac{1}{2}+\frac{1}{n}, & \mbox{if}~ n\neq m^2~\forall m\in \mathbb{N} ,\\
		t&~\mbox{ for all }~t\in \mathbb{N}. 
		%	\end{cases}
\end{align*}
Here, we  set $\mathcal{I}=\mathcal{I}_{\delta}$. Then, 
it is clear that $\{\omega_t\}_{t\in\mathbb{N}}$ is not statistically bounded. 
Now, observe that
$$(\omega_t,\mathcal{I}_\delta)-\text{LIM}^{r}x_t=
\begin{cases}
	\varnothing&~\mbox{for }~r\in[0,1), \\ 
	\{0\}&~\mbox{for }~ r\in[1,\infty).
\end{cases}$$ 
We fix any $r\geq 1$. Then for every $\beta,\sigma\geq0,$ we obtain that
$$(\omega_t,\mathcal{I}_\delta)-\text{LIM}^{r}x_t + \sigma B_X \nsubseteq (\omega_t,\mathcal{I}_\delta)-\text{LIM}^{\beta} x_t.$$
\end{example}

Our findings suggest that set $(\omega_t, \mathcal{I})-\text{LIM}^r x_t$ can remain non-empty if the roughness degree `$r$' is decreased under specific conditions.

\begin{theorem}\label{Th 2.1}
	Suppose $r,\sigma> 0$ and the sequence $\{x_t\}_{t\in\mathbb{N}}$ is contained in a totally bounded subset of $X$.  If $x_*\in X$ is such that $x_*+\sigma B_ X\subseteq (\omega_t,\mathcal{I})-\text{LIM}^{r} x_t$ then there exists $\bar\omega>0$ such that $r\geq\sigma\bar\omega$ and  $x_*\in (\omega_t,\mathcal{I})-\text{LIM}^{r-\sigma\bar\omega} x_t$. 
\end{theorem}
\begin{proof} Since $(\omega_t, \mathcal{I})-\text{LIM}^r x_t$ contains the non-trivial ball $x_*+\sigma B_ X$, Lemma \ref{le 1} entails that $\{\omega_t\}_{t\in\mathbb{N}}$ is $\mathcal{I}$-bounded. Then there exists $\mu>0$ such that 
	$$
	T=\{t\in\mathbb{N}: \omega_t<\mu\}\in\mathcal{F}(\mathcal{I}).
	$$ 
Let us set $\bar\omega=\sup\{\lambda\in\mathbb{R}:\omega_t\geq \lambda ~ \mbox{for all }~ t\in\mathbb{N}\}>0.$ Then $\omega_t\geq \bar{\omega}$ for each $t\in\mathbb{N}$. Since $x_*+\sigma B_ X\subseteq (\omega_t, \mathcal{I})-\text{LIM}^r x_t$, utilizing Lemma \ref{le 2}, we obtain  $2r\geq 2\sigma~\underset{t\in \mathbb{N}}{\liminf \omega_{t}}$. This ensures that $r\geq \sigma \bar\omega.$

Now, let $Y$ be a totally bounded subset of $X$ such that $\{x_t-x_*: t\in\mathbb{N}\}\subseteq Y$. Let $\varepsilon>0$. Then there exists $m\in\mathbb{N}$ with $d_1,d_2,...,d_m\in X\setminus \{0\}$ such that $Y\displaystyle\subseteq \bigcup_{j=1}^{m} (d_j+ \frac{\varepsilon}{3\mu}{B}_X)$.\\
	Consider the set $$
	D=\left \{-\frac{\sigma d_j}{\|d_j\|}:j=1,2,...,m \right\} \subseteq \displaystyle \sigma B_X.
	$$
	Note that, for each $z\in D,$ we have $x_*+z\in (\omega_t, \mathcal{I})-\text{LIM}^r x_t$. Therefore, for each $z\in D$, we get
	$$
	T_{z,\varepsilon}=\{t\in\mathbb{N}: \omega_t\|x_t-x_*-z\|<r+\frac{\varepsilon}{3}\}\in \mathcal{F}(\mathcal{I}).
	$$
	We set $T_\varepsilon=\displaystyle\bigcap_{z\in D} T_{z,\varepsilon}$. Since $D$ is finite, we also have  $T_\varepsilon\in \mathcal{F}(\mathcal{I})$.
	
	Pick $t \in T\cap  T_{\varepsilon}$ be arbitrary.  Since $x_t-x_*\in Y$, there exists $j\in \{1,2,...,m\}$ such that 
	 $$\|x_t-x_*-d_j\|\leq \frac{\varepsilon}{3\mu}.$$ 
	 
	Therefore, observe that
	\begin{align*}
		\omega_t(\sigma+\|d_j\|)=\omega_t(\|d_j \frac{\|d_j\|+\sigma}{\|d_j\|} \|)&=\omega_t(\|d_j+\frac{\sigma d_j}{\|d_j\|}\|)\\
		&\leq \sup_{z\in D} \mbox{~} \omega_t\|d_j-z\|\\
		&\leq  \sup_{z\in D} \mbox{~} (\omega_t\|x_t-x_*-z\|+\omega_t\|x_t-x_*-d_j\|)\leq r+ \frac{2\varepsilon}{3}.\\
		\Rightarrow \omega_t\|d_j\|&\leq r-\sigma\bar\omega +  \frac{2\varepsilon}{3}~\mbox{ (since }~ \omega_t\geq\bar\omega~\mbox{ for all }~ t\in\mathbb{N}).
	\end{align*}
	 With this, we have
	\begin{align*}
		\omega_t\|x_t-x_*\|&\leq \omega_t\|x_t-x_*-d_j\| + \omega_t\|d_j\|\\
		&\leq r-\sigma\bar\omega +\varepsilon
	\end{align*}
	Thus we obtain that $ T\cap  T_{\varepsilon}\subseteq\{t\in\mathbb{N}: \omega_t\|x_t-x_*\|\leq r-\sigma\bar\omega +\varepsilon \}$. This ensures that $\{t\in\mathbb{N}: \omega_t\|x_t-x_*\|> r-\sigma\bar\omega +\varepsilon\}\in\mathcal{I}$ since $T\cap  T_{\varepsilon}\in \mathcal{F}(\mathcal{I})$.
As $\varepsilon>0$ was chosen arbitrarily, we can conclude that $x_*\in (\omega_t,\mathcal{I})-\text{LIM}^{r-\sigma\bar\omega} x_t$. 
\end{proof}

  In Lemma \ref{le3}, we demonstrated that set $(\omega_t,\mathcal{I})-\text{LIM}^{r} x_t$ is a convex subset of $X$. An important question arises: is the set $(\omega_t,\mathcal{I})-\text{LIM}^{r} x_t$ strictly convex? Before we proceed, let us define some key results that are essential for proving the strict convexity of set $(\omega_t,\mathcal{I})-\text{LIM}^{r} x_t.$

  %Theorem \ref{strict convex}.

 According to \cite[Theorem 3.2]{kostyrko2000convergence}\label{usual subsequence}, for any P-ideal $\mathcal{I}$ on $\mathbb{N},$ a sequence $\{x_t\}_{t\in\mathbb{N}}$ is $\mathcal{I}$-convergent to $x$ if and only if there exists $A\subseteq\mathbb{N}$ with $\mathbb{N}\setminus A\in \mathcal{I}$ such that $\displaystyle\lim_{t\in A}x_t=x.$ We will apply a similar result for rough $(\omega_t,\mathcal{I})$-convergence, without providing a proof,  in light of \cite[Theorem 3.2(i)]{kostyrko2000convergence}

% Th 2.2
\begin{lemma}\label{Th 2.2}
	Let $r\geq 0$, and $\mathcal{I}$ be a $P$-ideal on $\mathbb{N}$. 
	Then for a sequence $\{x_t\}_{t\in\mathbb{N}}$ in $X$, $x_*\in (\omega_t,\mathcal{I})-\text{LIM}^{r} x_t$ if and only if there exists $A\subseteq \mathbb{N}$ with $\mathbb{N}\setminus A\in\mathcal{I}$ such that
	 $$\displaystyle\limsup_{t\in A}\omega_t\|x_t-x_*\|\leq r.$$
  \end{lemma}

\begin{definition} \cite{megginson2012introduction, xuan2003rough} Let $(X,\|\cdot\|)$ be a normed space.
	\begin{itemize} 
\item [(i)]  A subset $C\subseteq X$ is  strictly convex  if its interior contains the open line segment connecting any two points inside it.
\item[(ii)] $X$ is uniformly convex if for each $\varepsilon\in (0,2]$ there exists $\delta>0$ such that for all $x,y\in X$ satisfying $\|x\|=1,\|y\|=1$ and $\|x-y\|\geq\varepsilon,$ we have $\|\frac{x+y}{2}\|<\delta.$ 
\end{itemize} 
\end{definition}

%At this point, let us recall the following equivalent criterion of uniform convexity of a normed space $X$ in terms of  sequences with values in $X$, which plays a pivotal role in proving Theorem \ref{strict convex}.

\begin{lemma}\cite[Lemma 2.3 (a) \& (d)]{xuan2003rough}\label{Th 1.2*} $X$ is uniformly convex if and only if for each $r>0,$ and all $\varepsilon\in(0,2r]$ there exists $\delta(\varepsilon)>0$ such that, for arbitrary  sequences  $\{z_{0t}\}_{t\in\mathbb{N}}$  and  $\{z_{1t}\}_{t\in\mathbb{N}}$ in $X,$ we have
\begin{equation*}
	\limsup_{t\to\infty}\|z_{0t}\|\leq r,~ \limsup_{t\to\infty}\|z_{1t}\|\leq r,~\mbox{ and }~\|z_{0t}-z_{1t}\|\geq\varepsilon~\mbox{ for all }~t\in\mathbb{N}.
\end{equation*}
implies $\displaystyle\limsup_{t\to\infty}{\frac{1}{2}\|z_{0t}+z_{1t}\|}\leq r-\delta(\varepsilon).$
\end{lemma}

\begin{theorem}\label{strict convex}
Let $r\geq 0$, $\mathcal{I}$ be a $P$-ideal, and $\{x_t\}_{t\in\mathbb{N}}$ be a sequence with values in a  uniformly convex normed space $X$. Then the set $(\omega_t,\mathcal{I})-\text{LIM}^{r} x_t$   is strictly convex in $X$. 
\end{theorem}

\begin{proof}
	If $(\omega_t,\mathcal{I})-\text{LIM}^{r} x_t$ is empty or singletone then we are done. So we assume that $(\omega_t,\mathcal{I})-\text{LIM}^{r} x_t$ contains more than one element. Therefore, Lemma \ref{le 1} ensures that $\{\omega_t\}_{t\in\mathbb{N}}$ is $\mathcal{I}$-bounded. Then, there exists $\mu>0$ such that 
	$$
	T=\{t\in\mathbb{N}: \omega_t<\mu\}\in\mathcal{F}(\mathcal{I}).
	$$ 
Observe that to prove $(\omega_t,\mathcal{I})-\text{LIM}^{r} x_t$ is strictly convex, it is enough to show that for any $y_1,y_2\in(\omega_t,\mathcal{I})-\text{LIM}^{r} x_t$ with $\|y_1-y_2\|>0$, we have
	$$
	\hat{y}:=\frac{1}{2}(y_1+y_2)\in int((\omega_t,\mathcal{I})-\text{LIM}^{r} x_t).
	$$
	%because for each $y_{\lambda},0<\lambda<1,$ there exist $y_1',y_2'\in Wst\mbox{-}\text{LIM}^rx_i $ satisfying $y_1'\neq y_2'$ and $y_\lambda=\frac{1}{2}(y_1'+y_2').$\\
	Since $y_1,y_2\in(\omega_t,\mathcal{I})-\text{LIM}^{r} x_t$, Lemma \ref{Th 2.2} entails that
	there exists $A=\{t_1<t_2<...\}\subseteq\mathbb{N}$ with $A\in\mathcal{F}(\mathcal{I})$ such that \begin{equation*} \limsup_{k\to\infty}\|z_{1,t_k}\|\leq r,~\mbox{ and }~ \limsup_{k\to\infty}\|z_{2,t_k}\|\leq r, \end{equation*}
	where $z_{1,k}=\omega_k(y_1-x_k)$ and $z_{2,k}=\omega_k(y_2-x_k)$.\\
	Since $\{\omega_t\}_{t\in\mathbb{N}}$ is a weighted sequence, there exists $\lambda>0$ such that
	$\omega_t>\lambda$ for each $t\in\mathbb{N}.$ Thus, we have 
	\begin{equation*}\|z_{1,t_k}-z_{2,t_k}\|
		=\omega_{t_k}\|y_1-y_2\|> \lambda\|y_1-y_2\|=\varepsilon_*, \mbox{ for some }\varepsilon_*>0.
	\end{equation*}
	From Lemma \ref{le 2}, we get $\|y_1-y_2\|\leq  \frac{2r}{\lambda}$ i.e., $\varepsilon_* \in (0,2r].$  Now, in view of Lemma {\ref{Th 1.2*}}, there exists $\delta(\varepsilon_*)>0$ such that
	\begin{equation*}\limsup_{k\to\infty}\omega_{t_k}\|\frac{1}{2}(y_1+y_2)-x_{t_k}\|=\limsup_{k\to\infty}\frac{1}{2}\|z_{1,t_k}+z_{2,t_k}\|\leq r-\delta(\varepsilon_*).\end{equation*}
	Let $\varepsilon>0$ be given. Then, by the definition of $limsup$, there exists $k'\in\mathbb{N}$ such that 
	\begin{equation*}\omega_{t_k}\|\frac{1}{2}(y_1+y_2)-x_{t_k}\|<r-\delta(\varepsilon_*)+\varepsilon~\mbox{ for all }~k>k'.\end{equation*}
	Subsequently, 
	\begin{equation}\label{Eq 2}
		\{k\in\mathbb{N}:\omega_k\|\frac{1}{2}(y_1+y_2)-x_{k}\|<r-\delta(\varepsilon_*) +\varepsilon\}\supseteq A\setminus\{t_1,t_2,...,t_{k'}\}.
	\end{equation}
    Let us set
	 $$
	 C=\{k\in\mathbb{N}:\omega_k\|\frac{1}{2}(y_1+y_2)-x_{k}\|<r-\delta(\varepsilon_*) +\varepsilon\}\cap T
	 .$$ 
	 Evidently, we will have $C\in\mathcal{F}(\mathcal{I})$.
	We fix an $\varepsilon_0>0$ such that $\mu\varepsilon_0<\delta(\varepsilon_*)$ (we select $\mu$ from the set $T$).	Now, pick arbitrary $y\in \hat{y}+\varepsilon_0B_X$ where $	\hat{y}:=\frac{1}{2}(y_1+y_2).$ Then, for each $k\in C,$ we can write
	\begin{align*}
		\omega_k\|x_k-y\|&\leq \omega_k\|y-\hat{y}\|+ \omega_k\|\hat{y}-x_k\|\\
		&\leq \mu\varepsilon_0+r-\delta(\varepsilon_*)+\varepsilon< r+\varepsilon.
	\end{align*}
	Therefore, $\{k\in\mathbb{N}:\omega_k\|x_k-y\|>r+\varepsilon\}\subseteq \mathbb{N}\setminus C$, i.e., $\{k\in\mathbb{N}:\omega_k\|x_k-y\|>r+\varepsilon\}\in\mathcal{I}$. Consequently, $y\in (\omega_t,\mathcal{I})-\text{LIM}^{r} x_t$.  
Thus we obtain  $\hat{y}+\varepsilon_0B_X\subseteq(\omega_t,\mathcal{I})-\text{LIM}^{r} x_t$, i.e., $\hat{y}\in int((\omega_t,\mathcal{I})-\text{LIM}^{r} x_t)$.  Hence we deduce that  $(\omega_t,\mathcal{I})-\text{LIM}^{r} x_t$ is strictly convex.
 \end{proof}

The uniform convexity of the normed space in  Theorem \ref{strict convex} cannot be generally relaxed. We provide the following example to substantiate  our claim.
\begin{example} \label{ex 2.12}
	Consider any $P$-ideal $\mathcal{I}$ of $\mathbb{N}.$
	Let $\ell^\infty(\mathbb{R})$ be the normed  space consisting of  bounded sequences $\{x_t\}_{t\in\mathbb{N}}$ of real numbers endowed with $sup$ norm, i.e.,
	$$\|\{x_1,x_2,...\}\|_{\infty}:=\displaystyle\sup_{t\in\mathbb{N}}|x_t|.$$ 
	Note that $\ell^\infty(\mathbb{R})$  is not uniformly convex. Now, consider the sequence $\{e_t\}_{t\in\mathbb{N}},$ where $e_k$ represents the sequence $(0,...,0,1,0,...)$, with the $k$-th place being $1$ and the remaining positions being $0$.
	The weighted sequence $\{\omega_t\}_{t\in\mathbb{N}}$ is defined as follows:
	\begin{equation*}
		\omega_t=2+\frac{1}{t},~\mbox{ for each }~t\in\mathbb{N}.
	\end{equation*}
	Evidently, we have $\{e_1,e_2,e_3,...\}\subseteq (\omega_t,\mathcal{I})-\text{LIM}^{2} e_t$. We will show that $\hat e:=\frac{1}{2}(e_1+e_2)$ is not an interior point of $(\omega_t,\mathcal{I})-\text{LIM}^{2}e_t$. For each $\eta>0,$ we set $x_*=\{x_t\}_{t\in\mathbb{N}}\in\ell^\infty(\mathbb{R})$ such that 
	\begin{align*}
		x_{t}:=
		\begin{cases}
			~\frac{1}{2}~&\text{ for }~t\in \{1,2\}, \\
			\frac{-\eta}{2}~&\text{ otherwise}.
		\end{cases}
\end{align*}	 
	Then, it is clear that  
	$$
	x_*\in \hat{e}+{\eta}{B}_X~ \mbox{ and } ~ \|e_t-x_*\|=1+\frac{\eta}{2},~\mbox{ for all }~t\geq 3.
	$$
	If $0<\varepsilon_*<\frac{\eta}{2}$ then $\{t\in\mathbb{N}:\omega_t\|e_t-x_*\|>2+ \varepsilon_* \}\notin\mathcal{I}$, i.e., $x_*\notin (\omega_t,\mathcal{I})-\text{LIM}^{2} e_t$. Thus we can conclude that $\hat e\notin int ((\omega_t,\mathcal{I})-\text{LIM}^{2} e_t)$ since
	  for each $\eta>0$ there exists $x_*\in {B}_{\eta}(\hat{e})$ such that $x_*\notin (\omega_t,\mathcal{I})-\text{LIM}^{2} e_t$.
	
\end{example}

%A little observation may be helpful to understand the next corollary. Consider an open ball $B_\varepsilon(y)$ in a normed space $(X,\|.\|)$. Then $y$ can be written as an average of $y-x$ and $y+x$, where $x\in B_\varepsilon(0).$
\begin{corollary}\label{Cor 2.2}
	Suppose $\mathcal{I}$ is a $P$-ideal and $X$ is uniformly convex. If $y\in int((\omega_t,\mathcal{I})-\text{LIM}^{r} x_t)$ then there exists $s\in[0,r)$ such that $y\in(\omega_t,\mathcal{I})-\text{LIM}^{s} x_t$. 
\end{corollary}
\begin{proof}
	Since $y$ is an interior point of $(\omega_t,\mathcal{I})-\text{LIM}^{r} x_t$, there exist $y_1,y_2\in (\omega_t,\mathcal{I})-\text{LIM}^{r} x_t$ with $y_1\neq y_2$ such that $y=\frac{1}{2}(y_1+y_2)$. Now, in view of Eq (\ref{Eq 2}) of Theorem \ref{strict convex}, we obtain that
	$$
	\{k\in\mathbb{N}:\omega_k\|\frac{1}{2}(y_1+y_2)-x_{k}\|>r-\delta(\varepsilon_*) +\varepsilon\}\in\mathcal{I}.
	$$
	 As a consequence, $y\in (\omega_t,\mathcal{I})-\text{LIM}^{s} x_t$ with $s=r-\delta(\varepsilon_*)$.
\end{proof}

\section{Minimal convergent degree-related observation for $(\omega_t,\mathcal{I})-\text{LIM}^{r} x_t$.}\label{sec 3}
Here, we continue our quest along the same lines as in \cite{phu2001rough,xuan2003rough} which in certain instances enable us to present some classifications of the set $(\omega_t,\mathcal{I})-\text{LIM}^{r} x_t$, notably in reflexive normed spaces and uniformly convex Banach spaces when the degree of roughness `$r$' is minimal. Before proceeding further, let us introduce the concept of the \textit{minimal weighted ideal convergent degree} $\tilde{r}_{(\omega_t,\mathcal{I})}(x)$ of a sequence $x=\{x_t\}_{t\in\mathbb{N}}$ w.r.t. a given weighted sequence $\{\omega_t\}_{t\in\mathbb{N}}$ and a given ideal $\mathcal{I}$.
\begin{eqnarray}\label{2}
	\tilde{r}_{(\omega_t,\mathcal{I})}(x):=\inf\{r\in(0,\infty):(\omega_t,\mathcal{I})-\text{LIM}^{r} x_t\neq\varnothing\}.
\end{eqnarray}
Note that, in light of the definition of rough $(\omega_t, \mathcal{I})$-convergence, one can readily derive the following results:
\begin{equation}\label{eq1}
	(\omega_t,\mathcal{I})-\text{LIM}^{r_1} x_t\subseteq (\omega_t,\mathcal{I})-\text{LIM}^{r_2} x_t,~\mbox{ whenever }~ 0\leq r_1<r_2,
\end{equation}
\begin{equation}\label{Cor 2.3}
(\omega_t,\mathcal{I})-\text{LIM}^{r} x_t=\bigcap_{k\in\mathbb{N}} (\omega_t,\mathcal{I})-\text{LIM}^{(r+\frac{1}{k})} x_t ~\mbox{ for all }~r\geq 0.
\end{equation}
Observe that the monotonicity established in Eq (\ref{eq1}) entails that  
\begin{equation}\label{Eqn 2.3}
		(\omega_t,\mathcal{I})-\text{LIM}^{r} x_t
	\begin{cases}
		=\varnothing & \mbox{~for }~r<\tilde{r}_{(\omega_t,\mathcal{I})}(x), \\
		\neq\varnothing & \mbox{~for }~r>\tilde{r}_{(\omega_t,\mathcal{I})}(x).\\
	\end{cases} 
\end{equation}
We simply denote $\tilde{r}_{(\omega_t,\mathcal{I})}(x)$ by $\tilde{r}(x)$ when there is no confusion in the notations.\\

%%%%%%%%%%%%%%%%%%%%%%

In Equation \ref{Eqn 2.3}, it is naturally questioned whether if $r=\tilde{r}(x),$ the set $	(\omega_t,\mathcal{I})-\text{LIM}^{r} x_t$ is empty or not. We present the results from Example \ref{non-reflexive}, Theorem \ref{Th 2.6}, and Theorem \ref{Th 2.14} to answer this question.

\begin{example}\label{non-reflexive}	Consider the ideal $\mathcal{I}_{\delta}$ of natural density zero subsets of $\mathbb{N}$, and the normed space $C[0,1]$ of continuous functions on $[0,1]$ equipped with the integral norm, i.e., $$\|f\|=\int_{0}^{1}|f(u)| du ~\mbox{ for each } ~f\in C[0,1].$$
	The sequences $x=\{x_t\}_{t\in\mathbb{N}}$ and $\{\omega_t\}_{t\in\mathbb{N}}$ are defined, respectively, as follows:
	$$x_t(u)=
	\begin{cases}
		0 &~\mbox{ if }~0\leq u\leq\frac{1}{2},\\ 
		t(u-\frac{1}{2}) &~\mbox{ if }~ \frac{1}{2}<u<\frac{1}{2}+\frac{1}{t},\\
		1 &~\mbox{ if }~ \frac{1}{2}+\frac{1}{t}\leq u\leq 1.
	\end{cases}$$  
	and $$\omega_t=
	\begin{cases}
		\frac{1+t}{t} &~\mbox{ if }~t\notin\{n^2:n\in\mathbb{N}\},\\
		~t^2 &~ \mbox{ if }~t\in\{n^2:n\in\mathbb{N}\}.
	\end{cases}$$
	It is evident that the sequence $\{x_t\}_{t\in\mathbb{N}}$ does not possess any convergent subsequences (note that $x_n\xrightarrow{\|.\|} x$ ensures that
	\begin{equation*} x(u)=
		\begin{cases}
			0&~\mbox{for }~0\leq u<\frac{1}{2},\\
			1&~\mbox{for }~\frac{1}{2}<u\leq 1,
	\end{cases}\end{equation*}
	$-$ which is a discontinous function). We intend to show that $(\omega_t,\mathcal{I}_{\delta})-\text{LIM}^{0} x_t=\varnothing.$ Assume, on the contrary, that there exists $x_*\in C[0,1]$ such that  $x_*\in (\omega_t,\mathcal{I}_{\delta})-\text{LIM}^{0} x_t$. Since $\mathcal{I}_{\delta}$ is a $P$-ideal, in light of Lemma \ref{Th 2.2}, there exists $A\subseteq\mathbb{N}$ with $\delta(A)=1$ such that 
	$$
	\lim_{t\in A} \|x_t-x_*\|\leq \limsup_{t\in A} \omega_t\|x_t-x_*\| =0~\mbox{ (since  $\omega_t\geq 1$ for each $t\in\mathbb{N}$)},
	$$
	$-$ which contradicts the fact that $\{x_t\}_{t\in\mathbb{N}}$ has no subsequential limits. Therefore, we deduce that $(\omega_t,\mathcal{I}_{\delta})-\text{LIM}^{0} x_t=\varnothing$.

	Next, observe that  $\{x_t\}_{t\in\mathbb{N}}$ is Cauchy in $C[0,1]$. Choose an $r>0$. Then
	\begin{equation*}\|x_{t_1}-x_{t_2}\|<r~\mbox{ whenever }~t_1,t_2>\left[\frac{1}{r}\right]+1.
	\end{equation*}
	Let $\varepsilon>0$ be arbitrary. Note that, from the construction of $\{\omega_t\}_{t\in\mathbb{N}},$ it follows that there exists  $t_0\in\mathbb{N}$ such that
	$$\omega_t<1+\frac{\varepsilon}{r} ~\mbox{ for each }~t\in\{k\in\mathbb{N}:k>t_0\} \setminus\{1^2,2^2,3^2,...\}.$$
	Let us set $t_r=\max\{i_r,t_0\}$, where $i_r=\left[\frac{1}{r}\right]+1$. Therefore, we have
	\begin{equation*}
		\{t\in\mathbb{N}:\omega_t\|x_t-x_{i_r}\|\leq r+\varepsilon\}\supseteq \{t\in\mathbb{N}:t> t_r\}\setminus \{1^2,2^2,3^2,...\}.
	\end{equation*}
	Since $\delta(\{1^2,2^2,3^2,...\})=0$, we obtain $\delta(\{t\in\mathbb{N}:\omega_t\|x_t-x_{i_r}\|>r+\varepsilon\})=0$. Since $\varepsilon>0$ was arbitrary, we deduce that
	$x_{i_r}\in (\omega_t,\mathcal{I}_{\delta})-\text{LIM}^{r} x_t$. Therefore, 
	$(\omega_t,\mathcal{I}_{\delta})-\text{LIM}^{r}x_t\neq\varnothing$ for every  $r>0$. Thereon $\tilde{r}(x)=0$, but 
	$(\omega_t,\mathcal{I}_{\delta})-\text{LIM}^{0} x_t=\varnothing.$	
\end{example}

%%%%%%%%%%%%%%%%%%%%%%%%%%%%%%%%%%%%%%%

\begin{proposition}\label{Th 2.4}
	Let $x=\{x_t\}_{t\in\mathbb{N}}$ be a sequence with values in $X$, and $\{\omega_t\}_{t\in\mathbb{N}}$ be an $\mathcal{I}$-bounded weighted sequence.
	If $r>\tilde{r}(x)$ then $int((\omega_t,\mathcal{I})-\text{LIM}^{r} x_t)\neq \varnothing$. 
\end{proposition}
\begin{proof}
	Assume that $r>\tilde{r}(x).$ Let us pick $\sigma\in(0,r-\tilde{r}(x)),$ i.e., $r-\sigma>\tilde{r}(x).$ Then Eq (\ref{Eqn 2.3}) entails that $(\omega_t,\mathcal{I})-\text{LIM}^{r-\sigma} x_t\neq\varnothing$.  So pick an $x_*\in (\omega_t,\mathcal{I})-\text{LIM}^{r-\sigma} x_t.$  Let $\varepsilon>0$ be arbitrary.
	Then we have 
	  $$ 	A=\{t\in\mathbb{N}:\omega_t\|x_t-x_*\|>r-\sigma+\varepsilon\}\in\mathcal{I}.
	  $$ 
	  Since $\{\omega_t\}_{t\in\mathbb{N}}$ is $\mathcal{I}$-bounded, there exists $\mu>0$ such that 
	  $$
	  B=\{t\in\mathbb{N}:\omega_t>\mu\}\in\mathcal{I}.
	  $$
	Let us pick any  $y\in x_*+\frac{\sigma}{\mu}{B}_X$. Now observe that for each $t\in\mathbb{N}\setminus( A\cup B)$, we obtain
	\begin{align*}
		\omega_t\|x_t-y\|&\leq \omega_t\|x_*-y\|+\omega_t\|x_t-x_*\|\\
		&\leq \frac{\sigma\omega_t}{\mu}+\omega_t\|x_t-x_*\|\leq r+\varepsilon.
	\end{align*}
	Therefore, $\{t\in\mathbb{N}:\omega_t\|x_t-y\|>r+\varepsilon\}\subseteq A\cup B$.  Since $A\cup B\in\mathcal{I}$, we obtain that $y\in (\omega_t,\mathcal{I})-\text{LIM}^{r} x_t$. Thus, we can conclude that $x_*\in int((\omega_t,\mathcal{I})-\text{LIM}^{r} x_t)$.
\end{proof}
Now an important result is evinced, which is gleaned from Theorem \ref{Th 2.1}.

\begin{proposition}\label{Pro 3}
	Suppose  $x=\{x_t\}_{t\in\mathbb{N}}$ is contained in some totally bounded subset of $X$. Then $int((\omega_t,\mathcal{I})-\text{LIM}^{\tilde{r}(x)} x_t)=\varnothing$. 
\end{proposition}
\begin{proof} 
	First, assume that $\{\omega_t\}_{t\in\mathbb{N}}$ is not $\mathcal{I}$-bounded. Then, by Lemma \ref{le 1}, $(\omega_t,\mathcal{I})-\text{LIM}^{r} x_t$ has atmost one element. So we are done. Now, suppose that $\{\omega_t\}_{t\in\mathbb{N}}$ is $\mathcal{I}$-bounded. Assume, on the contrary, that $int((\omega_t,\mathcal{I})-\text{LIM}^{\tilde{r}(x)} x_t)\neq\varnothing$. So there exist $y\in X$ and $\sigma>0$  such that $y+\sigma {B}_X\subseteq (\omega_t,\mathcal{I})-\text{LIM}^{\tilde{r}(x)} x_t$. Therefore, by Theorem \ref{Th 2.1},  there exists $r'<\tilde{r}(x)$ such that $(\omega_t,\mathcal{I})-\text{LIM}^{r'} x_t\neq\varnothing-$ which contradicts the minimality of $\tilde{r}(x)$. Consequently, $(\omega_t,\mathcal{I})-\text{LIM}^{\tilde{r}(x)} x_t$ has no interior point.
\end{proof}

Let us now recall the definition of a reflexive normed space.
\begin{definition}\cite{megginson2012introduction}
A normed space  $X$ is reflexive if the canonical embedding map $x \mapsto x^{**}$ from $X$ to its double dual $ X^{**}$ is surjective. This map is defined by $ x^{**}(x^*) = x^*(x)$ for each $x^* \in X^*$, where $X^* $ denotes the dual of $X$.
\end{definition}
The following characterization of reflexive normed spaces will be utilized in Theorem \ref{Th 2.6}.
\begin{lemma} \cite[Theorem 1.13.6]{megginson2012introduction}\label{smulian}
A normed space $X$ is reflexive if and only if  $\displaystyle\bigcap_{k\in\mathbb{N}} C_k\neq\varnothing$ whenever $\{C_k\}_{k\in\mathbb{N}}$ is a sequence of non-empty, closed, bounded and convex subsets of $X$ such that $C_k\supseteq C_{k+1}$ for each $k$.
\end{lemma}
The subsequent theorem ensures the existence of the \textit{minimal weighted ideal convergent degree} for any $(\omega_t,\mathcal{I})$-convergent sequence in reflexive normed spaces. 

\begin{theorem}\label{Th 2.6}
For any sequence $x=\{x_t\}_{t\in\mathbb{N}}$ with values in a reflexive normed space $X$,  $(\omega_t,\mathcal{I})-\text{LIM}^{\tilde{r}(x)} x_t\neq\varnothing$.
\end{theorem}

\begin{proof} 
	Let us set $C_k=(\omega_t,\mathcal{I})-\text{LIM}^{(\tilde{r}(x)+\frac{1}{k})} x_t$ for each $k\in\mathbb{N}$. Then, in view of Lemma \ref{le 2} and Lemma \ref{le3}, we obtain that $\{C_k\}_{k\in\mathbb{N}}$ is a sequence of closed bounded convex subsets of $X$.
	Note, by Eq (\ref{eq1}) and Eq (\ref{Eqn 2.3}), that
	 $$
	C_{k+1}\subseteq C_{k}~\mbox{ and }~ C_k\neq\varnothing~ \mbox{ for each } k\in\mathbb{N}.
	 $$  
	Therefore, by Lemma \ref{smulian}, we can infer that  $\displaystyle\bigcap_{k\in\mathbb{N}} (\omega_t,\mathcal{I})-\text{LIM}^{(\tilde{r}(x)+\frac{1}{k})} x_t\neq\varnothing.$ Finally, in view of Eq (\ref{Cor 2.3}), we conclude that $(\omega_t,\mathcal{I})-\text{LIM}^{\tilde{r}(x)} x_t\neq\varnothing.$
\end{proof}

In Example \ref{non-reflexive}, we have demonstrated that the conclusion of the previous result does not always hold in arbitrary normed spaces.

\begin{corollary}
Suppose the sequence $x=\{x_t\}_{t\in\mathbb{N}}$ is bounded in some finite dimensional normed space $X$. Then for $r=\tilde{r}(x)$, 
	$$
(\omega_t,\mathcal{I})-\text{LIM}^{r} x_t\neq\varnothing,~\mbox{ and }~int((\omega_t,\mathcal{I})-\text{LIM}^{r} x_t)=\varnothing.
	$$ 
	The converse holds true if $\{\omega_t\}_{t\in\mathbb{N}}$ is $\mathcal{I}$-bounded.
\end{corollary}
\begin{proof}
Assume that $r=\tilde{r}$.
	Since $X$ is finite dimensional, $X$ must be reflexive.  As a result, Theorem \ref{Th 2.6} entails that $(\omega_t,\mathcal{I})-\text{LIM}^{r} x_t\neq\varnothing$. Note that $\{x_t\}_{t\in\mathbb{N}}$ is always contained in some totally bounded set since $\{x_t\}_{t\in\mathbb{N}}$ is bounded and $X$ is finite dimensional. Then, by Proposition \ref{Pro 3}, we conclude that
	$int((\omega_t,\mathcal{I})-\text{LIM}^{r} x_t)=\varnothing.$

	Conversely, let  $\{\omega_t\}_{t\in\mathbb{N}}$ be $\mathcal{I}$-bounded, and assume that there exists $r\geq 0$ such that  	
	$$
	(\omega_t,\mathcal{I})-\text{LIM}^{r} x_t\neq\varnothing,~\mbox{ and }~int((\omega_t,\mathcal{I})-\text{LIM}^{r} x_t)=\varnothing.
	$$ 
	Then, in view of Proposition \ref{Th 2.4}, we must have  $r\leq\tilde{r}(x).$ Also, it is evident that $r\geq \tilde{r}(x)$; otherwise we will have
	$(\omega_t,\mathcal{I})-\text{LIM}^{r} x_t=\varnothing$. Therefore, we deduce that $r=\tilde{r}(x)$.
\end{proof}

We conclude this section by characterizing the set $(\omega_t, \mathcal{I})-\text{LIM}^{\tilde r(x)} x_t$ in the context of uniformly convex Banach spaces.

\begin{theorem}\label{Th 2.14}
Let $r\geq0$ and $x=\{x_t\}_{t\in\mathbb{N}}$ be a sequence with values in a uniformly convex Banach space $X$.
Then the following results hold:
\begin{itemize}
\item[(\textbf{a1})] If $(\omega_t,\mathcal{I})-\text{LIM}^{r} x_t$ is singletone and $\{\omega_t\}_{t\in\mathbb{N}}$ is $\mathcal{I}$-bounded  then $r=\tilde{r}(x)$.
\item [(\textbf{a2})] If $r=\tilde{r}(x)$ and $\mathcal{I}$ is a $P$-ideal then $(\omega_t,\mathcal{I})-\text{LIM}^{r} x_t$ is singleton.
\end{itemize}
\end{theorem}
\begin{proof}
	(\textbf{a1}) First, assume that $(\omega_t,\mathcal{I})-\text{LIM}^{r} x_t$ is singleton and $\{\omega_t\}_{t\in\mathbb{N}}$ is $\mathcal{I}$-bounded. Then, in view of Eq (\ref{Eqn 2.3}), we obtain that $r\geq \tilde{r}(x)$. Clearly,
	$r>\tilde{r}(x)$ is not possible, otherwise, by Proposition \ref{Th 2.4}, $(\omega_t,\mathcal{I})-\text{LIM}^{r} x_t$ would contain a non-trivial ball. Thus we must have $r=\tilde{r}(x)$.

	(\textbf{a2}) We now assume that  $r=\tilde{r}(x)$ and $\mathcal{I}$ is a $P$-ideal. 
	Since every uniformly convex Banach space is reflexive (see \cite[Theorem 5.2.15]{megginson2012introduction}),
	 Theorem \ref{Th 2.6} entails that
	$(\omega_t,\mathcal{I})-\text{LIM}^{\tilde r(x)} x_t\neq \varnothing$.
	We claim that  $int((\omega_t,\mathcal{I})-\text{LIM}^{\tilde r} x_t)=\varnothing$ because 
	$int((\omega_t,\mathcal{I})-\text{LIM}^{\tilde r} x_t)\neq\varnothing$ ensures, by Corollary \ref{Cor 2.2}, that there exists  $s<\tilde{r}(x)$ such that $(\omega_t,\mathcal{I})-\text{LIM}^{s} x_t\neq\varnothing-$ which contradicts the minimality of $\tilde r(x)$.
	Finally, due to strict convexity of $(\omega_t,\mathcal{I})-\text{LIM}^{\tilde r(x)} x_t$ (see Theorem \ref{strict convex}), we conclude that $(\omega_t,\mathcal{I})-\text{LIM}^{\tilde r(x)} x_t$ is singleton. \end{proof}

\section{Rough weighted ideal cluster points.}\label{sec 4}
In this section, we introduce the concept of rough weighted ideal cluster points of a sequence in normed spaces and 
show the novelty of this notion by presenting an example (see Example \ref{ex 2}) that highlights the different behavior of rough weighted ideal cluster points compared to rough statistical cluster points of a sequence. 
Next, we provide a characterization of maximal admissible ideals through the notions of rough weighted ideal limit set and rough weighted ideal cluster point set of a sequence. Furthermore, several fascinating set theoretic results concerning rough weighted ideal cluster point set of a sequence are presented.
 \\
%We conclude this section by some interesting resutls related to rough weighted cluster points.

%Before moving further, let us recall the concept of rough statistical cluster points of a sequence.

%\begin{definition}\cite[Definition 2.1]{aytar2017rough}
%	Suppose  $r\geq 0$ and $x=\{x_t\}_{t\in\mathbb{N}}$ is a sequence in $X.$ Then $\gamma\in X$ is referred to as a rough statistical  (or, shortly $r$-statistical) cluster point of $x$ with degree of roughness  `$r$' if for each $\varepsilon>0,$ we have	
	%%$$
	%\delta(\{t\in\mathbb{N}:\|x_t-\gamma\|<r+\varepsilon\})\neq 0.
	%$$  
	%The set of $r$-statistical cluster points of $x$ is symbolized by $\Gamma_{x}^{r}.$
%\end{definition}
Before we proceed, let us define the concept of rough weighted ideal cluster points of a sequence in normed spaces. 

\begin{definition}\label{Def 2.2} 
For a sequence $\{x_t\}_{t\in\mathbb{N}}$ with values in a normed space $X$, we say that $\gamma\in X$ is  a rough weighted ideal cluster point of  $\{x_t\}_{t\in\mathbb{N}}$  w.r.t. a given weighted sequence $\{\omega_t\}_{t\in\mathbb{N}}$ and a  non-trivial ideal $\mathcal{I}$ (briefly, $(\omega_t,\mathcal{I})$-cluster point)
with degree of roughness $r\geq 0,$ if 
$$\{t\in\mathbb{N}:\omega_t\|x_t-\gamma\|<r+\varepsilon\}\notin\mathcal{I}~\mbox{ for every }~\varepsilon>0.$$
Denote by $(\omega_t,\mathcal{I})-\Gamma_{x_t}^r$ the set of rough $(\omega_t,\mathcal{I})$-cluster points of $\{x_t\}_{t\in\mathbb{N}}$ with degree of roughness $r$.
If \(\omega_{t} = 1\) for all \(t \in \mathbb{N}\) and \(\mathcal{I} = \mathcal{I}_{\delta}\), then one derives the definition of  rough statistical cluster points as delineated in \cite{aytar2017rough}. Additionally, by setting $r = 0$, we reach the definition of statistical cluster points \cite{fridy1993statistical}.
\end{definition}

\begin{theorem}\label{Non-empty} Let the weighted sequence  $\{\omega_t\}_{t\in\mathbb{N}}$ be $\mathcal{I}$-bounded and the sequence $\{x_t\}_{t\in\mathbb{N}}$ has $\mathcal{I}$-non-thin subsequence which lies on a totally bounded set of $X.$ Then for any  $r>0,$ $(\omega_t,\mathcal{I})-\Gamma_{x_t}^r\neq \varnothing.$
\end{theorem}

\begin{proof} Let $\{x_t\}_{t\in K}$ (where $K\not\in \mathcal{I}$) be a $\mathcal{I}$-non-thin subsequence of the sequence $\{x_t\}_{t\in\mathbb{N}}$ which lies on a totally bounded set $Y$ of $X.$  If possible let $(\omega_t,\mathcal{I})-\Gamma_{x_t}^{r_0}= \varnothing$ for some $r_0>0.$ If $\xi \in Y$ then $\xi \not\in (\omega_t,\mathcal{I})-\Gamma_{x_t}^r.$ Therefore there exists a $\varepsilon_{\xi}>0$ such that $$A(\varepsilon_{\xi})=\{t\in\mathbb{N}:\omega_t\|x_t-\xi\|<r+\varepsilon_{\xi}\}\in\mathcal{I}.$$
	From the $\mathcal{I}$-boundedness of the weighted sequence  $\{\omega_t\}_{t\in\mathbb{N}}$ there exists a positive real number $M$ such that $B=\{t\in \mathbb{N}: \omega_t<M\}\in \mathcal{F(\mathcal{I})}.$ This shows that $$B\cap (\mathbb{N}\setminus A(\varepsilon_{\xi}))\subseteq \{t\in\mathbb{N}:M\|x_t-\xi\|\geq r+\varepsilon_{\xi}\} \in \mathcal{F(\mathcal{I})}~\implies~ \{t\in\mathbb{N}:\|x_t-\xi\|<\frac{r+\varepsilon_{\xi}}{M}\}\in  \mathcal{I}.$$
	Let $\varepsilon_0=\displaystyle{\inf_{\xi \in Y}} \varepsilon_{\xi}.$ Then $\frac{r+\varepsilon_{0}}{M}>0.$ Therefore there exists $y_1,y_2,...,y_m$ such that
	
	$Y\subseteq \displaystyle{\bigcup_{i=1}^m} ~ B(y_i, \frac{r+\varepsilon_{0}}{M})~\implies ~ K=\{t\in \mathbb{N}:x_t\in Y\} \subseteq \displaystyle{\bigcup_{i=1}^m} \{t\in \mathbb{N}: \|x_t-y_i\|<\frac{r+\varepsilon_{0}}{M}\}\in \mathcal{I}.$ \end{proof}

\begin{remark} If we choose the weighted sequence  $\{\omega_t\}_{t\in\mathbb{N}}$ is not $\mathcal{I}$-bounded then the above theorem is not true. Let us consider the sequence $x=\{x_{t}\}_{t\in \mathbb{N}}=\{\frac{1}{t}\}_{t\in \mathbb{N}}$ lies on a totally bounded set as well as compact set $[0,1]$ of normed space $\mathbb{R}$ and the weighted sequence $\omega=\{\omega_t\}_{t\in\mathbb{N}}=\{t^2\}_{t\in \mathbb{N}}.$  
	Then  for any $r>0,$ and any ideal $\mathcal{I},$ we get $(\omega_t,\mathcal{I})-\Gamma_{x_t}^r=\varnothing.$
\end{remark}

Aytar proved in \cite[Theorem 2.2]{aytar2017rough} that  for a sequence $\{x_t\}_{t\in\mathbb{N}}$ in $X$, the rough statistical cluster point set is always closed in $X$. However, a notable deviation is observed in the case of $(\omega_t,\mathcal{I}) - \Gamma_{x_t}^r$, as presented in the next example.

%%%%%%%%%%%%%%%%%%%%%%%%%%%%%%%%%%%%%%%%%%%%%%%%%%%%%%%%%%%%%%%%%%%%%%%%%%%%%%%%%%%%%%%%%%%%%%%%%%%%%%%%%%%%%%%
\begin{example}\label{ex 2}
	Consider the normed space $(C[0,1],\|\cdot\|)$, where $\|f\|=\int_{0}^{1}|f(u)|du$ for all $f\in C[0,1]$. Now choose the ideal $\mathcal{I}_{\delta}$ of natural density zero subsets of $\mathbb{N}$. For each $j\in\mathbb{N}$, we set $A_{j}=\{2^{j-1}(2t-1):t\in\mathbb{N}\}\subset\mathbb{N}$. Then, notice that the sequence $\{A_t\}_{t\in\mathbb{N}}$ of sets is a partition of $\mathbb{N}$ such that $A_j\notin\mathcal{I}_{\delta}$ for any $j\in\mathbb{N}$. Let $r> 0$. We now define
	\begin{align*}
		x_{t}(u)= u^{j-1}+\frac{2ru}{t}~\mbox{ for }~t\in A_j,u\in[0,1]
		~~\text{ and }~~~
		\omega_{t}=
		%\begin{cases}
		%	\frac{1}{2}+\frac{1}{n}, & \mbox{if}~ n\neq m^2~\forall m\in \mathbb{N} ,\\
		t&~\mbox{ for all }~t\in \mathbb{N}. 
		%	\end{cases}
\end{align*}
Then $\{\omega_t\}_{t\in\mathbb{N}}$ is not statistically bounded. For each $j\in\mathbb{N}$, we set $g_{j}(u)=u^{j-1}$, where $u\in[0,1]$. Then, observe that 
$$
A_j\subseteq\{t\in\mathbb{N}:\omega_t\|x_t-g_j\|<r+\varepsilon\}\notin\mathcal{I}_{\delta}.
$$
Therefore, $g_j\in(\omega_t,\mathcal{I}_{\delta})-\Gamma_{x_t}^r$ for all $j\in\mathbb{N}$. Note that for each $k\in\mathbb{N}$, there exists $j\in\mathbb{N}$ such that $k=2^{j-1}(2t-1)$ for some $t\in\mathbb{N}$. Then
\begin{align*}
	\omega_k\|x_k-0\|=\omega_k\int_{0}^{1}(u^{j-1}+\frac{2ru}{k})du=\frac{2^{j-1}(2t-1)}{j}+r> \frac{1}{2}+r.
\end{align*}
Thus, we have
$$
\varnothing=\{t\in\mathbb{N}:\omega_t\|x_t-0\|<r+\frac{1}{2}\}\in\mathcal{I}_{\delta}.
$$
This ensures that $0\notin(\omega_t,\mathcal{I}_{\delta})-\Gamma_{x_t}^r$. Since  $g_j\xrightarrow{\|\cdot\|}0$, we conclude that $(\omega_t,\mathcal{I}_{\delta})-\Gamma_{x_t}^r$ is not closed in $C[0,1]$.
\end{example}

We will now establish criteria for both the weighted sequence $\{\omega_t\}_{t\in\mathbb{N}}$ and the ideal $\mathcal{I}$ that ensure $(\omega_t,\mathcal{I})-\Gamma_{x_t}^r$'s closedness, which will be elaborated on in the following two propositions.

\begin{proposition}\label{closedness cluster}
If the weighted sequence $\{\omega_t\}_{t\in\mathbb{N}}$ is $\mathcal{I}$-bounded then  $(\omega_t,\mathcal{I})-\Gamma_{x_t}^r$ is a closed subset of $X$.
\end{proposition}
\begin{proof}
	Assume that  $\{\omega_t\}_{t\in\mathbb{N}}$ is $\mathcal{I}$-bounded. Then  $A=\{t\in\mathbb{N}:\omega_t<\mu\}\in\mathcal{F}(\mathcal{I})$ for some $\mu>0$. Let $\{y_t\}_{t\in\mathbb{N}}$ be a sequence in $(\omega_t,\mathcal{I})-\Gamma_{x_t}^r$ such that $\displaystyle\lim_{t\to\infty}y_t=\ell$. Choose an $\varepsilon>0$. Then there exists $t_0\in\mathbb{N}$ such that $\|y_{t_0}-\ell\|<\frac{\varepsilon}{2\mu}$. Since $y_{t_0}\in (\omega_t,\mathcal{I})-\Gamma_{x_t}^r$, we obtain 
	$$
	B=\{t\in\mathbb{N}:\omega_t\|x_t-y_{t_0}\|<r+\frac{\varepsilon}{2}\}\notin\mathcal{I}.
	$$
	We set 
	$C=\{t\in\mathbb{N}:\omega_t\|x_t-\ell\|<r+\varepsilon\}$. Then, observe that $A\cap B\subseteq C$. Since $A\in\mathcal{F}(\mathcal{I})$ and $B\notin\mathcal{I}$, we have $A\cap B\notin\mathcal{I}$. Consequently, $C\notin\mathcal{I}$. From the arbitrariness of $\varepsilon>0$, we have $\ell\in (\omega_t,\mathcal{I})-\Gamma_{x_t}^r$. Hence we deduce that $(\omega_t,\mathcal{I})-\Gamma_{x_t}^r$ is closed in $X$.
\end{proof}
Recall that an ideal $\mathcal{I}$ on $\mathbb{N}$ is termed maximal if, for any ideal $\mathcal{J}$ on $\mathbb{N}$ satisfying $\mathcal{I} \subseteq \mathcal{J} \subseteq \mathbb{N}$, we have either \( \mathcal{J} = \mathcal{I} \) or $ \mathcal{J} = \mathbb{N}$. The following fact regarding maximal ideals will be utilized in this paper.

\begin{lemma}\cite[Lemma 5.1]{kostyrko2000convergence}\label{lemma maximal}
	Let $\mathcal{I}$ be an maximal admissible ideal on $\mathbb{N}$. Then for each $A\subset\mathbb{N}$, we have either $A\in\mathcal{I}$ or $\mathbb{N}\setminus A\in\mathcal{I}$.
\end{lemma}

\begin{proposition}\label{maximal closed cluster}
	If  $\mathcal{I}$ is a maximal admissible ideal on $\mathbb{N}$ then $(\omega_t,\mathcal{I})-\Gamma_{x_t}^r$ is closed in $X$.
\end{proposition}
\begin{proof}
If $\{\omega_t\}_{t\in\mathbb{N}}$ is $\mathcal{I}$-bounded, then by Proposition \ref{closedness cluster}, $(\omega_t,\mathcal{I})-\Gamma_{x_t}^r$ is closed in $X$. Let us suppose that $\{\omega_t\}_{t\in\mathbb{N}}$ is not $\mathcal{I}$-bounded. We claim that $(\omega_t,\mathcal{I})-\Gamma_{x_t}^r$ is either empty or singletone. Assume, on the contrary, that there exist $\gamma_1,\gamma_2\in (\omega_t,\mathcal{I})-\Gamma_{x_t}^r$  with $\gamma_1\neq\gamma_2$. Choose an $\varepsilon>0$. Since $\mathcal{I}$ is maximal admissible, note that 
$$\{t\in\mathbb{N}:\omega_t\|x_t-\gamma\|\geq  r+\varepsilon\}\in\mathcal{I}~\mbox{ for all }~\gamma\in
(\omega_t,\mathcal{I})-\Gamma_{x_t}^r.$$
We set $A=\{t\in\mathbb{N}:\omega_t>\frac{2(r+\varepsilon)}{\|\gamma_1-\gamma_2\|}\}$. Note that  $A\notin\mathcal{I}$ since  $\{\omega_t\}_{t\in\mathbb{N}}$ is not $\mathcal{I}$-bounded. Therefore, note that
\begin{eqnarray*}
	A&=&\{t\in\mathbb{N}:\omega_t\|\gamma_1-\gamma_2\|>2(r+\varepsilon)\}\\
	&\subseteq& \{t\in\mathbb{N}:\omega_t\|x_t-\gamma_1\|>r+\varepsilon\}\cup \{t\in\mathbb{N}:\omega_t\|x_t-\gamma_2\|>r+\varepsilon\}\in\mathcal{I}.
\end{eqnarray*}
This creats a contradiction, since $A\notin\mathcal{I}$. Hence the result.
\end{proof}

The converses of Propositions \ref{closedness cluster} and \ref{maximal closed cluster} are not true, as demonstrated in the following example.

	\begin{example}
	Consider the ideal $\mathcal{I}_{\delta}$ of natural density zero subsets of $\mathbb{N}$ and the normed space $B(\ell^2)$ of bounded linear operators on $\ell^2$. 
	The sequence $\{T_t\}_{t\in\mathbb{N}}$ in $B(\ell^2)$  is defined as follows:
	$$T_t(u)=
	\begin{cases}
		P_{t}(u) &~\mbox{ if }~t\in\{2n:n\in\mathbb{N}\},\\ 
		Q_t(u) &~\mbox{ if }~ t\in\{2n-1:n\in\mathbb{N}\},
	\end{cases}$$  
	where  $P_t,Q_t\in B(\ell^2)$ for each $t\in\mathbb{N}$, and for each $u=(u_1,u_2,...,u_t,...)\in\ell^2$, we have $$P_t(u)=(u_1,\frac{u_2}{2},...,\frac{u_t}{t},0,0,...)~\mbox{ and }~ Q_{t}(u)=(\underbrace{0,0,..,0}_{t~\text{zeros}},u_{t+1},u_{t+2},...).$$ 
	Next, we set the weighted sequence $\{\omega_t\}_{t\in\mathbb{N}}$ as follows:
	$$\omega_t=
	\begin{cases}
		\sqrt{t} &~\mbox{ if }~t\in\{2n:n\in\mathbb{N}\},\\
		~1 &~ \mbox{ if }~t\in\{2n-1:n\in\mathbb{N}\}.
	\end{cases}$$
Therefore,
$$\|P_{t}-P\|<\frac{1}{t+1}~\mbox{ and }~\|Q_t-0\|=1,$$
where $P:\ell^2\to\ell^2$ defined by $P(u)=(v_t)$ and $v_t=\frac{u_t}{t}$ for each $t\in\mathbb{N}$.
Let us fix $r>1$. Then we will have $(\omega_t,\mathcal{I}_{\delta})-\Gamma_{T_t}^r=\{0,P\}$. Evidently, we have $(\omega_t,\mathcal{I}_{\delta})-\Gamma_{T_t}^r$ is closed but neither the weighted sequence $\{\omega_t\}_{t\in\mathbb{N}}$ is statistically bounded nor the ideal $\mathcal{I}_{\delta}$ is maximal. \end{example}

The following theorem provides a crucial characterization of a maximal admissible ideal using the concepts of $(\omega_t,\mathcal{I})-\text{LIM}^{r}x_t$ and $(\omega_t,\mathcal{I})-\Gamma_{x_t}^r$.

\begin{theorem}\label{maximal ideal}
Let $r\geq 0$ and  $\mathcal{I}$ be an admissible ideal on $\mathbb{N}$. Then the ideal $\mathcal{I}$ is maximal if and only if $(\omega_t,\mathcal{I})-\text{LIM}^{r} x_t=(\omega_t,\mathcal{I})-\Gamma_{x_t}^r$ for every sequence $\{x_t\}_{t\in\mathbb{N}}$ in $X$.
\end{theorem}
\begin{proof}
	First, assume that $\mathcal{I}$ is maximal. Observe that 
	$
	(\omega_t,\mathcal{I})-\text{LIM}^{r} x_t\subseteq(\omega_t,\mathcal{I})-\Gamma_{x_t}^r
	$ holds true vacously for every $\{x_t\}_{t\in\mathbb{N}}$ in $X$. Now pick any $\gamma\in (\omega_t,\mathcal{I})-\Gamma_{x_t}^r$. Then 
	$$
	A=\{t\in\mathbb{N}:\omega_t\|x_t-\gamma\|<r+\varepsilon\}\notin\mathcal{I}~\mbox{ for every }~\varepsilon>0.
	$$
	Since $\mathcal{I}$ is maximal, in view of Lemma \ref{lemma maximal}, we  have $\mathbb{N}\setminus A=\{t\in\mathbb{N}:\omega_t\|x_t-\gamma\|\geq r+\varepsilon\}\in\mathcal{I}$.  Since $\varepsilon>0$ was arbitrary, we get that $\gamma\in 	(\omega_t,\mathcal{I})-\text{LIM}^{r} x_t$, i.e., $(\omega_t,\mathcal{I})-\Gamma_{x_t}^r\subseteq(\omega_t,\mathcal{I})-\text{LIM}^{r} x_t$.
	Consequently, $(\omega_t,\mathcal{I})-\text{LIM}^{r} x_t=(\omega_t,\mathcal{I})-\Gamma_{x_t}^r$.
	
	Next, suppose that $(\omega_t,\mathcal{I})-\text{LIM}^{r} x_t=(\omega_t,\mathcal{I})-\Gamma_{x_t}^r$ holds true for every sequence $\{x_t\}_{t\in\mathbb{N}}$ in $X$. Assume, on the contrary, that $\mathcal{I}$
is not maximal. Then there exists $A\subseteq\mathbb{N}$ such that either $A,\mathbb{N}\setminus A\in\mathcal{I}$ or $A,\mathbb{N}\setminus A\notin\mathcal{I}$. Since $\mathcal{I}$ is admissible, $\mathcal{I}$ is non-trivial (by definition). Therefore, we must have $A,\mathbb{N}\setminus A\notin\mathcal{I}$, otherwise $\mathbb{N}\in\mathcal{I}$.
Let $\{s_t\}_{t\in\mathbb{N}}$ be a sequence in $X$ such that $\|s_t\|=1$ for each $t\in\mathbb{N}$. Let us define
$$x_t=
\begin{cases}
	\frac{rs_t}{\omega_t} &~\mbox{ if }~t\in A,\\
	ts_t &~ \mbox{ if }~t\in\mathbb{N}\setminus A.
\end{cases}$$
From the construction of $\{x_t\}_{t\in\mathbb{N}}$, we have $0\in (\omega_t,\mathcal{I})-\Gamma_{x_t}^r$ but $0\notin (\omega_t,\mathcal{I})-\text{LIM}^{r} x_t$, which is a contradiction. Thus we can conclude that $\mathcal{I}$ is a maximal ideal.
\end{proof}
\begin{corollary}
There exists a sequence $\{x_t\}_{t\in\mathbb{N}}$ in $X$ such that 
$$
(\omega_t,\mathcal{I}_{\delta})-\text{LIM}^{r} x_t\subsetneq(\omega_t,\mathcal{I}_{\delta})-\Gamma_{x_t}^r.
$$
\end{corollary}
\begin{proof}
	Note that 
	$(\omega_t,\mathcal{I}_{\delta})-\text{LIM}^{r} x_t\subseteq(\omega_t,\mathcal{I}_{\delta})-\Gamma_{x_t}^r$ always holds true for every $\{x_t\}_{t\in\mathbb{N}}$ in $X$.
	Since $\mathcal{I}_{\delta}$ is not a maximal ideal, the result follows from Theorem \ref{maximal ideal}.
\end{proof}
As previously noted  whenever $\{\omega_t\}_{t\in\mathbb{N}}$ is $\mathcal{I}$-bounded, $(\omega_t,\mathcal{I})-\Gamma_{x_t}^r$ is closed (see Proposition \ref{closedness cluster}).
In this context, we offer a representation of closed sets in separable normed spaces in relation to $(\omega_t, \mathcal{I})-\Gamma_{x_t}^r$.
\begin{theorem}\label{closed set}
Let $X$ be a separable normed space and $\{\omega_t\}_{t\in\mathbb{N}}$ be $\mathcal{I}$-bounded. Suppose there exists a disjoint sequence $\{A_t\}_{t\in\mathbb{N}}$ of sets such that $A_t\subset\mathbb{N}$ and $A_t\notin\mathcal{I}$ for any $t\in\mathbb{N}$. Then for every non-empty closed set $F\subseteq X$ there exists a sequence $\{x_t\}_{t\in\mathbb{N}}$ in $X$ such that 
$$
F=\bigcap_{r>0}(\omega_t,\mathcal{I})-\Gamma_{x_t}^r.
$$
\end{theorem}
\begin{proof}
Since $X$ is a separable and subspace of a separable metric space is separable, there exists a countable set $A=\{a_1,a_2,...,a_k,..\}\subset F$ such that $\bar{A}=F$.  If $\bigcup_{i=1}^{\infty}A_i=\mathbb{N}$, we define
$$
x_{t}=
a_i ~\mbox{ if }~t\in A_i.
$$
Otherwise, we define
$$x_t=
\begin{cases}
	a_i &~\mbox{ if }~t\in A_i,\\
	a_1 &~ \mbox{ if }~t\in\mathbb{N}\setminus \bigcup_{i=1}^{\infty}A_i.
\end{cases}$$
Since $\{\omega_t\}_{t\in\mathbb{N}}$ is a weighted sequence, there exists $\beta>0$ such that $\omega_t>\beta$ for all $t\in\mathbb{N}$. Additionally, since the sequence $\{\omega_t\}_{t\in\mathbb{N}}$ is $\mathcal{I}$-bounded, there exists  $\mu > 0$ 
such that 
$$
C=\{t\in\mathbb{N}:\omega_t<\mu\}\in \mathcal{F}(\mathcal{I}).
$$  
Now, choose an $\varepsilon>0$.
First, assume that $z\in F$. Then there exists $a_i\in A$ such that $\|z-a_i\|<\frac{\varepsilon}{2\mu}$. Now, for all $r>0$, observe that 
$$
\{t\in\mathbb{N}:\omega_t\|x_t-a_i\|<r+\frac{\varepsilon}{2}\}\cap C\subseteq \{t\in\mathbb{N}:\omega_t\|x_t-z\|<r+\varepsilon\}.
$$
Note that $\{t\in\mathbb{N}:\omega_t\|x_t-a_i\|<r+\frac{\varepsilon}{2}\}\supseteq A_i$. Since $A_i\notin\mathcal{I}$ and $C\in \mathcal{F}(\mathcal{I})$,
we obtain $\{t\in\mathbb{N}:\omega_t\|x_t-a_i\|<r+\frac{\varepsilon}{2}\}\cap C\notin\mathcal{I}$. This implies $\{t\in\mathbb{N}:\omega_t\|x_t-z\|<r+\varepsilon\}\notin \mathcal{I}$. This yields $z\in \displaystyle\bigcap_{r>0}(\omega_t,\mathcal{I})-\Gamma_{x_t}^r$, i.e., $F\subseteq \displaystyle\bigcap_{r>0}(\omega_t,\mathcal{I})-\Gamma_{x_t}^r$.\\
Next, assume that $z\in \displaystyle\bigcap_{r>0}(\omega_t,\mathcal{I})-\Gamma_{x_t}^r$. 
%If $y=a_i$ for some $i\in\mathbb{N}$ then $y\in F$. So, let $y\neq a_i$ for any $i\in\mathbb{N}$. 
 Then 
$$
\{t\in\mathbb{N}:\omega_t\|x_t-z\|<r+\frac{\beta\varepsilon}{2}\}\notin\mathcal{I}~\mbox{ for every }~r>0.
$$
In particular, for $r=\frac{\beta\varepsilon}{2}>0$, we have 
$$
B=\{t\in\mathbb{N}:\omega_t\|x_t-z\|<\beta\varepsilon\}\notin\mathcal{I}.
$$
Since $B\notin\mathcal{I}$, we get that $B\neq\varnothing$. So, choose any $j\in B$. Therefore, we have 
\begin{align*}
	\beta\|x_j-z\|&<\omega_j\|x_j-z\|<\beta\varepsilon\\
	&\Rightarrow \|x_j-z\|<\varepsilon.
\end{align*}
Since $x_j=a_i$ for some $i\in\mathbb{N}$ and $\varepsilon>0$ was arbitrary, we conclude that $z\in \bar{A}=F$, i.e.,  $\displaystyle\bigcap_{r>0}(\omega_t,\mathcal{I})-\Gamma_{x_t}^r\subseteq F$. Thus we deduce that
  $
  F=\displaystyle\bigcap_{r>0}(\omega_t,\mathcal{I})-\Gamma_{x_t}^r.
  $
\end{proof}
\begin{corollary}
	If $X$ is separable and $\{\omega_t\}_{t\in\mathbb{N}}$ is statistically bounded then for each non-empty closed set $F\subseteq X$ there exists a sequence $\{x_t\}_{t\in\mathbb{N}}$ in $X$ such that 
	$$
	F=\bigcap_{r>0}(\omega_t,\mathcal{I}_{\delta})-\Gamma_{x_t}^r.
	$$
\end{corollary}
\begin{proof}
 Since $A_{j}=\{2^{j-1}(2t-1):t\in\mathbb{N}\}\notin\mathcal{I}_{\delta}$ and $A_i\cap A_j=\varnothing$ if $i\neq j$, the result directly follows from Theorem \ref{closed set}. 
\end{proof}
%%%%%%%%%%%%%%%%%%%%%%%%%%%%%%%%%%%%%%%%%%%%%%%%%%%%%%%%%%%%%%%%%%%%%%%%%%%%%%%%%%%%%%%%%%%%%%%%%%%%%%%%%%%%%%
%Th 2.6
We conclude this section with  a fascinating  characterization of the set $(\omega_t,\mathcal{I})-\Gamma_{x_t}^r$.
\begin{proposition}
Suppose $\mathcal{I}$ is a non-trivial ideal on $\mathbb{N}$ and $\{x_t\}_{t\in\mathbb{N}}$ is a sequence in $X$. We set $A_{\ell}=\{\omega_t(x_t-\ell):t\in A\}$ where $\ell\in X$ and $A\in\mathcal{F}(\mathcal{I})$.
\begin{itemize}
 \item[(\textbf{a1})] Then $\ell\in (\omega_t,\mathcal{I})-\Gamma_{x_t}^r$ for all $r>0$ if and only if $0\in \bar{A_{\ell}}$ for all $A\in\mathcal{F}(\mathcal{I})$.
\item[(\textbf{a2})] If   $\ell\in (\omega_t,\mathcal{I})-\Gamma_{x_t}^r$ for some $r>0$ then $A_{\ell}\cap r'B_X\neq\varnothing$ for all  $A\in\mathcal{F}(\mathcal{I})$ where $r'>r$.
\end{itemize}
\begin{proof}
(\textbf{a1}) First, suppose that  $\ell\in (\omega_t,\mathcal{I})-\Gamma_{x_t}^r$ for all $r>0$. Let $\varepsilon>0$ be arbitrary. Then, 
$$
\{t\in\mathbb{N}:\omega_t\|x_t-\ell\|<r+\frac{\varepsilon}{2}\}\notin\mathcal{I}~\mbox{ for all }~r>0.
$$
Thus, for $r=\frac{\varepsilon}{2}>0$, we have $B=\{t\in\mathbb{N}:\omega_t\|x_t-\ell\|<\varepsilon\}\notin\mathcal{I}$. Let $A\in\mathcal{F}(\mathcal{I})$. Since $B\notin\mathcal{I}$, we have $A\cap B\neq\varnothing$. Therefore, there exists $t\in A$ such that $\omega_t\|x_t-\ell\|<\varepsilon$. This ensures that  $0\in \bar{A_{\ell}}$.

Conversely, let $0\in \bar{A_{\ell}}$ for all $A\in\mathcal{F}(\mathcal{I})$. Assume, on the contrary, that $\ell\notin (\omega_t,\mathcal{I})-\Gamma_{x_t}^r$ for some $r>0$. Then there exists $\varepsilon_*>0$ such that $\{t\in\mathbb{N}:\omega_t\|x_t-\ell\|\leq r+\varepsilon_*\}\in\mathcal{I}$. Consequently, 
$$
A=\{t\in\mathbb{N}:\omega_t\|x_t-\ell\|> r+\varepsilon_*\}\in\mathcal{F}(\mathcal{I}).
$$
With this $A\in\mathcal{F}(\mathcal{I})$, we conclude that $A_{\ell}=\{\omega_t(x_t-\ell):t\in A\}\subset X$ is such that 
$$
A_{\ell}\subseteq X\setminus (r+ \frac{\varepsilon_*}{2})B_X.
$$
This contradicts that $0\notin \bar{A_{\ell}}$. Therefore, we need to have $\ell\in (\omega_t,\mathcal{I})-\Gamma_{x_t}^r$ for all $r>0$.\\
(\textbf{a2}) Let $\ell\in (\omega_t,\mathcal{I})-\Gamma_{x_t}^r$. Since $r'>r$, there exists $\varepsilon>0$ such that $r+\varepsilon\leq r'$. Then $B=\{t\in\mathbb{N}:\omega_t\|x_t-\ell\|<r+\varepsilon\}\notin\mathcal{I}$. Now, choose any $A\in\mathcal{F}(\mathcal{I})$. Since $B\notin\mathcal{I}$, we have $A\cap B\neq\varnothing$. Pick any $t\in A\cap B$. Then we have  
$$
\omega_t\|x_t-\ell\|\leq r'~\mbox{ and }~t\in A.
$$
This ensures that $A_{\ell}\cap r'B_X\neq\varnothing$ for any $A\in\mathcal{F}(\mathcal{I})$.
\end{proof}
\end{proposition}

\section {Application to Korovkin type approximations via weighted equi-ideal convergence.}\label{section 5}

We begin this section by recalling the following fundamental definitions related to 
$\mathcal{I}$-convergence for sequences of functions, which will serve as the basis for our subsequent discussions and results.

\begin{definition}\cite[Section 2]{balcerzak2007statistical}
Let $\mathcal{I}$ be an admissible ideal and $\{f_t\}_{t\in\mathbb{N}}$ be a sequence of functions in $C(K)$. Then $\{f_t\}_{t\in\mathbb{N}}$ is said to be 
\begin{itemize}
\item[(i)] $\mathcal{I}$-pointwise convergent to $f\in C(K)$ if 
$$
(\forall x\in K)~(\forall \varepsilon>0)~\Rightarrow \{t\in\mathbb{N}:|f_t(x)-f(x)|> \varepsilon\}\in\mathcal{I}.
$$
\item [(ii)]$\mathcal{I}$-uniform convergent to $f\in C(K)$ if for every $\varepsilon>0$
$$
\{k\in\mathbb{N}:\|f_t-f\|_{C(K)}>\varepsilon\}\in\mathcal{I}.
$$
\end{itemize}
Additionally, for $\mathcal{I} = \mathcal{I}_\delta$, a new type of convergence called equi-statistical convergence, was introduced. This mode of convergence lies between $\mathcal{I}_\delta$-pointwise convergence and $\mathcal{I}_\delta$-uniform convergence, and is defined as follows:

A sequence $\{f_t\}_{t \in \mathbb{N}}$ is said to be equi-statistically convergent to $f \in C(K)$ if, for every $\varepsilon > 0$, the sequence $\{g_{j,\varepsilon}\}_{j \in \mathbb{N}}$ of functions $g_{j,\varepsilon} : K \to \mathbb{R}$, defined by
\begin{equation*}
g_{j,\varepsilon}(x) = \frac{|\{t \in \mathbb{N} : |f_t(x) - f(x)| > \varepsilon\} \cap [1,j]|}{j}, \quad x \in K,
\end{equation*}
converges uniformly to the zero function on $K$.
\end{definition}

It has been shown that the ideal $\mathcal{I}_{\delta}$ is an analytic $P$-ideal as it can be represented as $\mathcal{I}_{\delta}=Exh (\varphi)$ (see \cite[Proposition 1.1]{balcerzak2015generalized} with $g(t)=t$), where
$$
\varphi(A)=\sup_{n\in\mathbb{N}}\frac{|A\cap[1,t]|}{t}~\quad \mbox{ for all }~A\subseteq \mathbb{N}.
$$
Similarly, for a given weighted sequence $\{\omega_t\}_{t\in\mathbb{N}}$, it can be deduced that the ideal $\mathcal{I}_{\mathcal{W}\delta} = \{A \subset \mathbb{N} : \mathcal{W}\delta(A) = 0\}$, consisting of sets of weighted density zero, forms an analytic $P$-ideal, with $\mathcal{I}_{\mathcal{W}\delta} = Exh(\varphi)$, where
$$
\varphi(A)=\sup_{t\in\mathbb{N}}\frac{|A\cap[1,\theta_t]|}{\theta_t}~\quad \mbox{ for all }~A\subseteq \mathbb{N}.
$$
We now demonstrate that, for such specific choices of $\varphi$, weighted equi-$\mathcal{I}_{\varphi}$ convergence coincides with either equi-statistical convergence or weighted equi-statistical convergence, thereby reinforcing the validity of our proposed notion in Definition~\ref{defIequi}..
\begin{proposition}\label{procoincidence}
Let $\{\omega_t\}_{t\in\mathbb{N}}$ be a weighted sequence and $\varphi:\mathcal{P}(\mathbb{N})\to [0,\infty)$ be such that 
$$
\varphi(A)=\sup_{n\in\mathbb{N}}\frac{|A\cap[1,\theta_t]|}{\theta_t}~\quad \mbox{ for all }~A\subseteq \mathbb{N}.
$$
Then $\varphi$ is a submeasure on $\mathbb{N}$ and weighted equi-$\mathcal{I}_{\varphi}$ convergence coincides with weighted equi-statistical convergence. In particular,  this corresponds to equi-statistical convergence when $\omega_t=1$ for each $t\in\mathbb{N}$.
\end{proposition}
\begin{proof}
	It is straightforward to verify that $\varphi$ defines a submeasure on $\mathbb{N}$. For this submeasure $\varphi$, we establish the equivalence between weighted equi-$\mathcal{I}_{\varphi}$ convergence and weighted equi-statistical convergence. 
 First, suppose that  $\{f_t\}_{t \in \mathbb{N}}$ is  weighted equi-statistically convergent to $f \in C(K)$. Therefore, we obtain 
$$
(\forall \sigma>0)~(\exists j_0\in\mathbb{N})~(\forall j\geq j_0)~(\forall x\in K)\;\Rightarrow ~g_{j,\varepsilon}(x)<\sigma,
$$
where 
\begin{equation*}
	g_{j,\varepsilon}(x)=\frac{|\{t\leq \theta_j:\omega_t|f_t(x)-f(x)|> \varepsilon\}|}{\theta_j},~x\in K
\end{equation*}
Now $(\forall j\geq j_0)~(\forall x\in K)$, we have 
\begin{align*}
h_{j,\varepsilon}(x)&=\varphi(\{t\in\mathbb{N}:\omega_t|f_t(x)-f(x)|> \varepsilon\}\setminus[1,j])\\
&=\sup_{n>j}\frac{|(\{t\in\mathbb{N}:\omega_t|f_t(x)-f(x)|> \varepsilon\}\setminus[1,j])\cap [1,n]|}{n}\\
&\leq \sup_{n>j}\frac{|\{t\in\mathbb{N}:\omega_t|f_t(x)-f(x)|> \varepsilon\}\cap [1,n]|}{n}\\
&=\sup_{n>j}g_{n,\varepsilon}(x)\leq\sigma~\mbox{ (since }~n\geq j_0).
\end{align*}
This ensures that, for any $\varepsilon>0$, the sequence of functions $\{h_{j,\varepsilon}\}_{j\in\mathbb{N}}$ converges uniformly to $0$, i.e., $\{f_t\}_{t \in \mathbb{N}}$ is  weighted equi-$\mathcal{I}_\varphi$ convergent to $f$.

Next, assume that $\{f_t\}_{t \in \mathbb{N}}$ is  weighted equi-$\mathcal{I}_\varphi$ convergent to $f$. Fix any $\sigma>0$. Since  $\{h_{j,\varepsilon}\}_{j\in\mathbb{N}}$ converges uniformly to $0$, $(\exists j_0\in\mathbb{N})~(\forall x\in K)$ such that 
$$
h_{j_0,\varepsilon}(x)=\sup_{n>j_0}\frac{|(\{t\in\mathbb{N}:\omega_t|f_t(x)-f(x)|> \varepsilon\}\setminus[1,j_0])\cap [1,n]|}{n}<\frac{\sigma}{2}.
$$
Now choose $n_0> j_0$ such that 
$$
(\forall x\in K)~ ~(\forall n> n_0),~\frac{|\{t\in\mathbb{N}:\omega_t|f_t(x)-f(x)|> \varepsilon\}\cap[1,j_0]|}{n}<\frac{\sigma}{2}.
$$
Consequently, $(\forall n>n_0)~(\forall x\in K)$, we have 
\begin{align*}
g_{n,\varepsilon}(x)&=\frac{|\{t \in \mathbb{N} : \omega_t|f_t(x) - f(x)| > \varepsilon\} \cap [1,n]|}{n}\\
&=\frac{|\{t \in \mathbb{N} : \omega_t|f_t(x) - f(x)| >\varepsilon\} \cap [1,j_0]|}{n}+\frac{|(\{t \in \mathbb{N} : \omega_t|f_t(x) - f(x)| > \varepsilon\} \cap [1,n])\setminus [1,j_0]|}{n}\\
&<\frac{\sigma}{2}+h_{j_0,\varepsilon}(x)<\frac{\sigma}{2}+\frac{\sigma}{2}=\sigma.
\end{align*}
This shows that the sequence of functions $\{g_{j,\varepsilon}\}_{j\in\mathbb{N}}$ converges uniformly to $0$, i.e., $\{f_t\}_{t \in \mathbb{N}}$ is  weighted equi-statistically convergent to $f$. 
\end{proof}

The subsequent example reveals that the conclusion of  Theorem \ref{thMm5.4} and Example \ref{example5.5} do not remain valid under the choice of the weighted sequence $\omega_t=\sqrt{t}$ for each $t\in\mathbb{N}$. 
\begin{example}\label{exwrong}
Let us choose any $f\in C^1[0,1]$ such that $|f'(x)|>0$ for all $x\in C[0,1]$. We set 
$$
\omega_t=\sqrt{t}~\mbox{ for each }~t\in\mathbb{N}.
$$ 
Then, in light of Theorem \ref{tachev} with $q=1$, we get that 
$$
\omega_t\cdot \left( B_t(f; x) - f(x) - B_t((e_1-x);x)f'(x) \right) \to 0~\mbox{ uniformly on}~[0,1].
$$
This also ensures that 
\begin{equation}\label{eq 9}
	B_t(f; x) - f(x) -B_t((e_1-x);x)f'(x)\to 0~~\quad(\omega-equi-st)
\end{equation}
We aim to show that
$$
L_t(f;x)\not\to f~~\quad(\omega-equi-st).
$$
Assume, on the contrary, that 
\begin{equation}\label{eq 10}
L_t(f;x)-f(x)=(B_t(f;x)-f(x))+h_t(x)B_t(f;x)\to 0 ~~\quad(\omega-equi-st)
\end{equation}
Note that $h_t(x)\neq 0$ for all $x\in (\frac{1}{2^t},\frac{1}{2^{t-1}})$. 
So, for a fixed $x\in[0,1]$, the function $h_t(x)$ is non-zero for at most one $t\in \mathbb{N}$. Therefore, observe that 
$$
\frac{1}{\theta_n}|\{t\leq \theta _n:\omega_t|h_t(x)B_t(f;x)|>\varepsilon\}|\leq \frac{1}{\theta_n}\to 0 ~\mbox{ as }~n\to \infty.
$$
Thus, we obtain that 
 $$
 h_t(x)B_t(f;x)\to 0\quad(\omega-equi-st).
 $$
 Since $L_t(f;x)-f(x)\to 0 ~(\omega-equi-st)$, Eq (\ref{eq 10}) entails that 
 \begin{equation}\label{eq 11}
  B_t(f;x)-f(x)\to 0 ~~\quad(\omega-equi-st).
 \end{equation}
 Moreover, note that 
 $$
 B_t((e_1-x);x)f'(x)= B_t(f;x)-f(x)-\left( B_t(f; x) - f(x) -B_t((e_1-x);x)f'(x) \right)
 $$
\end{example} 
Therefore, in view of Eq (\ref{eq 9}) and Eq (\ref{eq 11}), it follows that 
$$
 B_t((e_1-x);x)f'(x)\to 0 ~~\quad(\omega-equi-st),
$$
- which is a contradiction since $B_t((e_1-x);x)=x(1-x)$ and  $|f'(x)|>0$ for all $x\in [0,1]$. Thus we can conclude that 
$$
L_t(f;x)\not\to f~~\quad(\omega-equi-st)~\mbox{ for all $f\in C^1[0,1]$ with $|f'(x)|>0$ for all $x\in [0,1]$}.
$$\\
We are now in a position to refine Theorem~\ref{thMm5.4} in the context of weighted equi-ideal convergence, which generalizes the notion of weighted equi-statistical convergence as well as equi-statistical convergence.
	\begin{theorem}	\label{korokvingequiideal}
		Let $\mathcal{I}_{\varphi}$ be an analytic $P$-ideal, and let $\{\omega_t\}_{t \in \mathbb{N}}$ be a weighted sequence that is $\mathcal{I}_{\varphi}$-bounded. Consider a compact subset $K \subset \mathbb{R}$, and let $\{L_t : C(K) \to C(K)\}_{t \in \mathbb{N}}$ be a sequence of positive linear operators.
		 Then, for all $f \in C(K)$, 
	\begin{equation*}
	L_t(f) \to f \quad \text{($\omega-equi-\mathcal{I}_{\varphi}$)} \quad \text{on } K \tag{a}
\end{equation*}
	if and only if
\begin{equation}
		L_t(e_i) \to e_i \quad \text{($\omega-equi-\mathcal{I}_{\varphi}$)} \quad \text{on } K \quad \text{with } e_i(x) = x^i, \ i = 0, 1, 2\tag{b}.
\end{equation}
\end{theorem}
	\begin{proof}
		The forward implication (a) $\Rightarrow$ (b) is immediate, since each monomial $e_i \in C(X)$ for $i = 0, 1, 2$.

To establish the converse, assume that condition (b) holds. Since the weighted sequence $\{\omega_{t}\}_{t\in\mathbb{N}}$ is $\mathcal{I}_{\varphi}$-bounded, there exists $\mu>0$ such that 
$$
A=\{t\in\mathbb{N}:\omega_t>\mu\}\in\mathcal{I}_{\varphi}.
$$
Now, fix an arbitrary function $f \in C(K)$ and a point $c \in K$. By the continuity of $f$ at $c$, for any $\varepsilon > 0$, there exists a $\delta > 0$ such that
\[
|f(x) - f(c)| \leq \frac{\varepsilon}{\mu} \quad \text{whenever } |x - c| < \delta.
\]

Define the set $K_\delta := [c - \delta, c + \delta] \cap K$ and denote by $\chi_D$ the characteristic function of a set $D$. Then we can decompose
\[
|f(x) - f(c)| = |f(x) - f(c)| \chi_{K_\delta}(x) + |f(x) - f(c)| \chi_{K \setminus K_\delta}(x).
\]

This leads to the bound
\[
|f(x) - f(c)| \leq \frac{\varepsilon}{\mu} + 2B \frac{(x - c)^2}{\delta^2} \quad \text{for all } x \in K,
\]
where $B := \|f\|_{C(K)}$.

Using the positivity and linearity of the operators $L_t$, we obtain
\[
|L_t(f; c) - f(c)| \leq \frac{\varepsilon}{\mu} + (\frac{\varepsilon}{\mu} + B) |L_t(e_0; c) - e_0(c)| + \frac{2B}{\delta^2} L_t\left((x - c)^2; c\right).
\]

Expanding the quadratic term,
\[
(x - c)^2 = x^2 - 2cx + c^2 = e_2(x) - 2c e_1(x) + c^2 e_0(x),
\]
and applying $L_t$ yields
\[
L_t\left((x - c)^2; c\right) = L_t(e_2; c) - 2c L_t(e_1; c) + c^2 L_t(e_0; c).
\]

Consequently,
\begin{align*}
|L_t(f; c) - f(c)| \leq 
&\frac{\varepsilon}{\mu}+ \left(\frac{\varepsilon}{\mu} + B + \frac{2B c^2}{\delta^2} \right) |L_t(e_0; c) - e_0(c)|\\& 
+ \frac{4B c}{\delta^2} |L_t(e_1; c) - e_1(c)| + \frac{2B}{\delta^2} |L_t(e_2; c) - e_2(c)|.
\end{align*}

Thus, we can write
\begin{equation*}
	|L_t(f; c) - f(c)| \leq \frac{\varepsilon}{\mu} + B' \sum_{i=0}^2 |L_t(e_i; c) - e_i(c)|, 
\end{equation*}
where
\[
B' := \frac{\varepsilon}{\mu} + B + \frac{2B}{\delta^2} \left( \|e_2\|_{C(K)} + 2\|e_1\|_{C(K)} + 1 \right).
\]
This ensures that
$$
\omega_t |L_t(f; c) - f(c)| \leq \varepsilon + B' \sum_{i=0}^2 \omega_t|L_t(e_i; c) - e_i(c)|~\mbox{ for each}~n\in\mathbb{N}\setminus A. 
$$
Let us set 
\begin{align*}
&\Psi_j(c,\varepsilon)=\varphi(\{t\in\mathbb{N}\setminus A:\omega_t |L_t(f; c) - f(c)|>\frac{\varepsilon}{2}\}\setminus [1,j])\\
&\Psi_{i,j}(c,\varepsilon)=\varphi(\{t\in\mathbb{N}\setminus A:\omega_t |L_t(e_i; c) - e_i(c)|>\frac{\varepsilon}{6B'}\}\setminus [1,j])~\mbox{ for }~i=0,1,2.
\end{align*}
Since $c\in K$ was chosen arbitrarily and  $\varphi$  is monotone non-decreasing, we obtain that
\begin{align*}
\Psi_j(x,\varepsilon)\leq \displaystyle \sum_{i=0}^{2}\Psi_{i,j}(x,\varepsilon)~\mbox{for all $x\in K$ and for each }~j\in\mathbb{N}.
\end{align*}
By the hypothesis, it follows that the sequence of functions $\{\Psi_{i,j}(x,\varepsilon)\}_{j\in\mathbb{N}}$ is uniformly convergent to zero function on $K$ for each $i=0,1,2$. As a consequence, $\{\Psi_j(x,\varepsilon)\}_{j\in\mathbb{N}}$ converges uniformly to zero function on $K$. It now remains to show that the sequence of functions $\{h_j(x,\varepsilon)\}_{j\in\mathbb{N}}$, defined by 
$$
h_j(x,\varepsilon)=\varphi(\{t\in\mathbb{N}:\omega_t |L_t(f; x) - f(x)|>\frac{\varepsilon}{2}\}\setminus [1,j]),
$$
also converges uniformly to zero function on $K$. This holds since
$$
0\leq h_j(x,\varepsilon)\leq \Psi_j(x,\varepsilon)+\varphi(A\setminus [1,j])~\mbox{ and }~A\in\mathcal{I}_{\varphi}.
$$
Hence the implication (b) $\Rightarrow$ (a) holds as well.
\end{proof}

%\section{{Declarations}}

%{\bf Ethical Approval} : Not applicable.\\

%{\bf Competing Interests} : We declare that the authors have no competing interests as defined by Springer, or other interests that might be perceived to influence the results and /or discussion reported in this paper.\\

%{\bf Authors' Contributions} : All the authors have equally contributed to this article.\\

%{\bf Funding} : The research of the first author is supported by Human Resource Development Group, CSIR, India, through NET-JRF and the grant number is 09/0285(12636)/2021-EMR-I.\\

%{\bf Data Availibility} : This manuscript has no associated data.\\

%{\bf Competing interests} :  The authors declare that they have no known competing financial interests or personal relationships that could have appeared to influence the work reported in this paper. 

%\section*{Acknowledgments} We are thankful to Prof. Pratulananda Das, Jadavpur University, West Bengal, India for several valuable suggestions which improved the quality and presentation of the paper.

%\section*{Declarations}

%{\bf Author contribution:} Both the authors have equal contributions in this work.\\

%{\bf Availability of data and material:} There are no data related to this paper.\\

%{\bf Conflict of interest:} The authors have no relevant financial or non-financial interests to disclose.\\

		\end{document}